\documentclass{amsart}
\usepackage[top=3cm,bottom=3cm,left=2.5cm,right=2.5cm]{geometry}

\usepackage{graphicx} 
\graphicspath{ {./images/} }
\usepackage[nobreak]{cite}
\usepackage{amsthm}
\usepackage{amsmath}
\usepackage{amssymb}
\usepackage{mathtools}
\usepackage{thm-restate}
\usepackage{enumitem}
\usepackage{tikz}
\usepackage{tikz-cd}
\usetikzlibrary{patterns,decorations.markings,decorations.pathreplacing}
\usepackage{hyperref}
\hypersetup{
    colorlinks=true,
    linkcolor=blue,
    citecolor=blue,
    urlcolor=blue
}
\usepackage[noabbrev,capitalize,nameinlink]{cleveref}
\usepackage{ulem}
\usepackage{relsize}
\usepackage{multirow}
\usepackage{multicol}
\usepackage{caption}
\usepackage{subcaption}
\usepackage{hhline}
\usepackage{makecell}
\usepackage{dashrule}
\usepackage[section]{placeins}
\usepackage{indentfirst}
\usepackage{needspace}

\definecolor{wongblack}{RGB}{0,0,0}
\definecolor{wongorange}{RGB}{230,159,0}
\definecolor{wonglightblue}{RGB}{86,180,233}
\definecolor{wonggreen}{RGB}{0,158,115}
\definecolor{wongyellow}{RGB}{240,228,66}
\definecolor{wongblue}{RGB}{0,114,178}
\definecolor{wongvermillion}{RGB}{213,94,0}
\definecolor{wongpurple}{RGB}{204,121,167}

\crefname{prop}{Proposition}{Propositions}
\Crefname{prop}{Proposition}{Propositions}
\crefname{theo}{Theorem}{Theorems}
\Crefname{theo}{Theorem}{Theorems}
\crefname{defi}{Definition}{Definitions}
\Crefname{defi}{Definition}{Definitions}

\theoremstyle{plain}
\newtheorem{theo}{Theorem}[section]
\newtheorem{prop}[theo]{Proposition}
\newtheorem{lemm}[theo]{Lemma}

\newtheorem{coro}[theo]{Corollary}

\theoremstyle{definition}
\newtheorem{defi}[theo]{Definition}
\newtheorem{rema}[theo]{Remark}

\newtheorem{exam}[theo]{Example}

\newcommand{\CC}{\mathbb{C}}
\newcommand{\R}{\mathbb{R}}

\newcommand{\Z}{\mathbb{Z}}

\newcommand{\even}{\mathrm{even}}
\newcommand{\odd}{\mathrm{odd}}

\newcommand{\FSA}{\mathrm{FSA}}
\newcommand{\mc}{\mathrm{MC}}

\DeclareMathOperator{\inv}{inv}

\DeclareMathOperator{\ic}{\widetilde{inv}}

\DeclarePairedDelimiter{\rbra}{\lparen}{\rparen} 
\DeclarePairedDelimiterX{\Set}[2]{\lbrace}{\rbrace}{#1\,\delimsize\vert\,#2}

\newcommand{\myemph}[1]{\textcolor{wongblue}{\textit{#1}}}

\title[A Probabilistic Bijection between 20V Configurations and Triple-Free GT Patterns]{A Probabilistic Bijection between Twenty-Vertex Configurations with a Free West Boundary and Gelfand--Tsetlin Patterns Avoiding Three Equal Entries in a Row}
\author[A. Yoshida]{Atsuro Yoshida}
\address[A. Yoshida]{Fakult\"at f\"ur Mathematik, Universit\"at Wien, Oskar-Morgenstern-Platz 1, Wien 1090, Austria}
\email{atsuro.yoshida@univie.ac.at}

\date{}



\begin{document}

\begin{abstract}
    We study a coincidence between two enumerations governed by the same product formula, reminiscent of the Robbins numbers: the unweighted enumeration of twenty-vertex configurations on quadrangular domains with fixed west boundary, and the weighted enumeration of Gelfand--Tsetlin patterns avoiding three equal entries in a row.
    This coincidence naturally raises the question of whether there is a combinatorial explanation relating these two enumerations.
    In this paper, we provide such an explanation by constructing a probabilistic bijection between twenty-vertex configurations on quadrangular domains and Gelfand--Tsetlin patterns avoiding three equal entries in a row.
    Under this probabilistic bijection, the west boundary of a twenty-vertex configuration corresponds to the bottom row of Gelfand--Tsetlin patterns; in particular, the fixed boundary case corresponds to Gelfand--Tsetlin patterns with bottom row~$(1, 2, \ldots, n)$.
    Combining this correspondence with an enumeration formula of Fischer and Schreier-Aigner for Gelfand--Tsetlin patterns avoiding three equal entries in a row with bounded entries, we obtain an enumeration formula for twenty-vertex configurations with a free west boundary.
\end{abstract}

\maketitle

\section{Introduction}
A \myemph{twenty-vertex model} (\myemph{$20$V model}) \cite{kelland1974twenty,MR998375} is an ice model on the triangular lattice whose edges are the horizontal and vertical edges of the square lattice, together with the diagonal edges connecting the northwest and southeast vertices of each square face.
A \myemph{twenty-vertex configuration} (\myemph{$20$V configuration}) is an orientation of the edges in a domain of this lattice such that every internal vertex is incident to exactly three incoming edges and three outgoing edges, a condition known as the \myemph{ice rule}.
\Cref{fig:20V-example-2} shows an example of a $20$V configuration.
\begin{figure}[htb]
    \centering
    \includegraphics[scale=0.8]{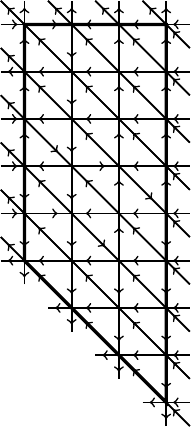}
    \caption{An example of a $20$V configuration on a quadrangular domain.}\label{fig:20V-example-2}
\end{figure}

In \cite{MR4395233}, Di Francesco studied twenty-vertex configurations on quadrangular domains.
One of the main results of that paper is that the number of $20$V configurations on a certain quadrangular domain of size $n$ coincides with the number of domino tilings of the so-called \myemph{Aztec triangle} of size $n$.
Moreover, Di Francesco conjectured that these numbers are given by the product formula
\begin{equation}
    1, 4, 60, 3328, 678912, \ldots
    = 2^{n(n-1)/2}\prod_{j=0}^{n-1} \frac{(4j+2)!}{(n+2j+1)!},
    \label{eq:DF-product-formula}
\end{equation}
which is reminiscent of the celebrated Robbins numbers
\[
\prod_{j=0}^{n-1} \frac{(3j+1)!}{(n+j)!},
\]
counting alternating sign matrices, totally symmetric self-complementary plane partitions, descending plane partitions, and alternating sign triangles \cite{MR865837,MR847558,MR700040,MR4081430}.

At the 9th International Conference on ``Lattice Path Combinatorics and Applications,'' Zeilberger announced that this conjecture had been proved by Koutschan.
The proof subsequently appeared as Section~4 of \cite{MR4773148}, co-authored with Krattenthaler and Schlosser.

Shortly thereafter, Fischer and Schreier-Aigner studied in \cite{MR4752187} a weighted enumeration of \myemph{Gelfand--Tsetlin patterns} (\myemph{GT patterns}) avoiding three equal entries in a row.
A GT pattern~$T=(T_{i,j})_{1 \leq i \leq n, 1 \leq j \leq i}$ is a triangular array of integers such that $T_{i+1, j} \leq T_{i, j} \leq T_{i+1, j+1}$ holds for all valid $i, j$.
In this paper, we refer to GT patterns avoiding three equal entries in a row (i.e., there is no $(i, j)$ with $T_{i,j}=T_{i,j+1}=T_{i,j+2}$) as \myemph{triple-free GT patterns}.
Each triple-free GT pattern~$T$ is assigned the weight~$\omega_{\FSA}(T)$ defined by $\omega_{\FSA}(T)=2^r$ where $r$ is the number of entries that are not equal to both their upper-left and upper-right neighbors.
Here is an example of a triple-free GT pattern.
\begin{center}
    \begin{tikzpicture}[scale=0.75]
        \node[outer sep=10pt] (1-1) at (1,1) {$2$};
        \node[outer sep=10pt] (2-1) at (2,1) {\tiny$\leq$};
        \node[outer sep=10pt, rectangle, draw] (3-1) at (3,1) {$4$};
        \node[outer sep=10pt] (4-1) at (4,1) {\tiny$\leq$};
        \node[outer sep=10pt] (5-1) at (5,1) {$5$};
        \node[outer sep=10pt] (6-1) at (6,1) {\tiny$\leq$};
        \node[outer sep=10pt] (7-1) at (7,1) {$8$};
        \node[outer sep=10pt] (8-1) at (8,1) {\tiny$\leq$};
        \node[outer sep=10pt] (9-1) at (9,1) {$9$};

        \node[outer sep=10pt, rotate=45] (1.5-1.5) at (1.5,1.5) {\tiny$\leq$};
        \node[outer sep=10pt, rotate=45] (3.5-1.5) at (3.5,1.5) {\tiny$\leq$};
        \node[outer sep=10pt, rotate=45] (5.5-1.5) at (5.5,1.5) {\tiny$\leq$};
        \node[outer sep=10pt, rotate=45] (7.5-1.5) at (7.5,1.5) {\tiny$\leq$};

        \node[outer sep=10pt, rotate=315] (2.5-1.5) at (2.5,1.5) {\tiny$\leq$};
        \node[outer sep=10pt, rotate=315] (4.5-1.5) at (4.5,1.5) {\tiny$\leq$};
        \node[outer sep=10pt, rotate=315] (6.5-1.5) at (6.5,1.5) {\tiny$\leq$};
        \node[outer sep=10pt, rotate=315] (8.5-1.5) at (8.5,1.5) {\tiny$\leq$};

        \node[outer sep=10pt] (2-2) at (2,2) {$4$};
        \node[outer sep=10pt] (3-2) at (3,2) {\tiny$\leq$};
        \node[outer sep=10pt] (4-2) at (4,2) {$4$};
        \node[outer sep=10pt] (5-2) at (5,2) {\tiny$\leq$};
        \node[outer sep=10pt, rectangle, draw] (6-2) at (6,2) {$6$};
        \node[outer sep=10pt] (7-2) at (7,2) {\tiny$\leq$};
        \node[outer sep=10pt] (8-2) at (8,2) {$9$};

        \node[outer sep=10pt, rotate=45] (2.5-2.5) at (2.5,2.5) {\tiny$\leq$};
        \node[outer sep=10pt, rotate=45] (4.5-2.5) at (4.5,2.5) {\tiny$\leq$};
        \node[outer sep=10pt, rotate=45] (6.5-2.5) at (6.5,2.5) {\tiny$\leq$};
        \node[outer sep=10pt, rotate=315] (3.5-2.5) at (3.5,2.5) {\tiny$\leq$};
        \node[outer sep=10pt, rotate=315] (5.5-2.5) at (5.5,2.5) {\tiny$\leq$};
        \node[outer sep=10pt, rotate=315] (7.5-2.5) at (7.5,2.5) {\tiny$\leq$};

        \node[outer sep=10pt] (3-3) at (3,3) {$4$};
        \node[outer sep=10pt] (4-3) at (4,3) {\tiny$\leq$};
        \node[outer sep=10pt] (5-3) at (5,3) {$6$};
        \node[outer sep=10pt] (6-3) at (6,3) {\tiny$\leq$};
        \node[outer sep=10pt] (7-3) at (7,3) {$6$};

        \node[outer sep=10pt, rotate=45] (3.5-3.5) at (3.5,3.5) {\tiny$\leq$};
        \node[outer sep=10pt, rotate=45] (5.5-3.5) at (5.5,3.5) {\tiny$\leq$};
        \node[outer sep=10pt, rotate=315] (4.5-3.5) at (4.5,3.5) {\tiny$\leq$};
        \node[outer sep=10pt, rotate=315] (6.5-3.5) at (6.5,3.5) {\tiny$\leq$};

        \node[outer sep=10pt] (4-4) at (4,4) {$5$};
        \node[outer sep=10pt] (5-4) at (5,4) {\tiny$\leq$};
        \node[outer sep=10pt] (6-4) at (6,4) {$6$};

        \node[outer sep=10pt, rotate=45] (4.5-4.5) at (4.5,4.5) {\tiny$\leq$};
        \node[outer sep=10pt, rotate=315] (5.5-4.5) at (5.5,4.5) {\tiny$\leq$};

        \node[outer sep=10pt] (5-5) at (5,5) {$5$};
    \end{tikzpicture}
\end{center}
In this example, the entries that are equal to both their upper-left and  neighbors are drawn boxed, and the weight~$\omega_{\FSA}(T)$ is computed as $2^{15-2}=2^{13}$.
Fischer and Schreier-Aigner showed that the weighted enumeration of triple-free GT patterns with $n$ rows and bottom row~$(1, 2, \ldots, n)$ with respect to this weight equals $2^n$ times \eqref{eq:DF-product-formula}.

Thus, there are currently three distinct classes of combinatorial objects known to be enumerated by the same product formula~\eqref{eq:DF-product-formula}, among which no explicit bijection is known.

The present paper has two purposes.
The first purpose is to provide a combinatorial explanation for the coincidence of two of the three classes of combinatorial objects having the same (weighted) enumeration formula, namely $20$V configurations and triple-free GT patterns, by using the concept of a \myemph{probabilistic bijection}.

A probabilistic bijection, as used in \cite{MR3874706,MR4031106,MR4311960,MR4818707}, is a probabilistic generalization of bijections between finite sets.
If there is a probabilistic bijection between two weighted finite sets, it follows that the weighted enumerations of the two sets are equal, similar to the case of ordinary bijections between two unweighted finite sets.
We construct the probabilistic bijection by composing two probabilistic bijections between the following objects:
\begin{enumerate}[label=(\alph*)]
    \item $20$V configurations on the quadrangular domain;
    \item \myemph{mixed six-vertex configurations} on the rectangle domain;
    \item triple-free GT patterns.
\end{enumerate}
The probabilistic bijection between (a) and (b) is obtained by interpreting \myemph{Yang--Baxter moves} in \cite{MR4395233} as probabilistic maps, which are then composed.
The probabilistic bijection between (b) and (c) is constructed by interpreting a certain surjective map as a probabilistic bijection.
By composing these two probabilistic bijections, we obtain a probabilistic bijection between the set of $20$V configurations and the set of triple-free GT patterns.

Fischer and Schreier-Aigner {\cite{MR4752187}} considered a non-negative parameter~$m$ as an upper bound of the entries on a GT pattern.
The second purpose of this paper is to reinterpret this parameter~$m$ in the setting of $20$V configurations.

Under the aforementioned probabilistic bijection, the positions of incoming horizontal edges on the west boundary of a $20$V configuration determine the bottom row of the corresponding triple-free GT patterns.
Since the entries of a GT pattern are bounded above by the rightmost entry in the bottom row, the parameter~$m$ determines the highest possible position of an incoming horizontal edge on the west boundary.

For example, in \cref{fig:20V-config-example}, the horizontal edges at positions~$2, 3, 4, 6$ on the west boundary are incoming, while the other horizontal edges are outgoing.
Hence, this $20$V configuration corresponds to triple-free GT patterns with bottom row $(2,3,4,6)$ through the probabilistic bijection.

This setting is more general than the fixed west boundary case studied in \cite{MR4245068,MR4395233}, where all incoming edges appear consecutively from bottom to top without gaps.

Consequently, we obtain the following result.

\begin{theo} \label{theo:extended-20V-almost-monotone-triangle-correspondence}
    Let $\mathbf{k} = (k_1, k_2, \ldots, k_n)$ be a strictly increasing sequence of positive integers.
    The number of $20$V configurations on the quadrangular domain~$\mathcal{Q}_\mathbf{k}$, where the horizontal boundary edges on the west boundary at positions~$k_1,k_2,\ldots,k_n$ are incoming, while the remaining horizontal boundary edges are outgoing, is equal to $2^{-n}$ times the weighted enumeration of triple-free GT patterns with bottom row~$\mathbf{k}$ with respect to the weight function~$\omega_{\FSA}(\cdot)$.
\end{theo}

Here, we emphasize that we do not use the explicit formula for the number of $20$V configurations with fixed west boundary nor the explicit formula for the weighted enumeration of triple-free GT patterns with fixed bottom row in the proof of this theorem; the proof is based solely on the construction of the probabilistic bijection.

It is also proven in \cite{MR4752187} that the weighted enumeration of triple-free GT patterns with bounded entries is given by a product formula.
Using our probabilistic bijection to transfer this result to $20$V configurations with a free west boundary, we obtain the following theorem.

\begin{theo} \label{theo:free-left-boundary-20V-enumeration}
    Let $m$ be a non-negative integer.
    We have
    \[
    \sum_{\substack{\mathbf{k}=(k_1,\dots,k_n) \\ 1 \le k_1 < \dots < k_n \le m+1}} \#\{\text{$20$V configurations on $\mathcal{Q}_\mathbf{k}$}\}
    = \prod_{i=1}^n \frac{(m-n+3i+1)_{i-1} (m-n+i+1)_{i}}{(\frac{m-n+i+2}{2})_{i-1} (i)_i},
    \]
    where $(a)_k \coloneqq a(a+1)(a+2) \cdots (a+k-1)$ is the Pochhammer notation.
\end{theo}

Setting $m = n-1$ in this theorem, $\mathbf{k}$ is forced to be $(1,2,\ldots,n)$ and we recover the product formula~\eqref{eq:DF-product-formula} for $20$V configurations with the horizontal boundary edges on the west boundary at positions~$1,2,\ldots,n$.

\textbf{This paper is organized as follows.}
In \cref{sec:combinatorial-objects}, we define $20$V configurations and the mixed $6$V configurations.
In \cref{sec:local-prob-map}, we introduce probabilistic bijections and construct a probabilistic bijection between $20$V configurations and mixed $6$V configurations using Yang--Baxter moves.
In \cref{sec:comb-desc}, we express the weights of mixed $6$V configurations in terms of an integer-valued statistic that we call the \myemph{variant inversion number}.
In \cref{sec:map-mixed-6V-to-amtri}, we construct a surjective map from mixed $6$V configurations to triple-free GT patterns and interpret it as a probabilistic bijection.
Finally, in \cref{sec:proof-of-main-results}, we combine these probabilistic bijections to prove \cref{theo:extended-20V-almost-monotone-triangle-correspondence,theo:free-left-boundary-20V-enumeration}.

\Needspace{8\baselineskip}
\section{Twenty-vertex and mixed six-vertex configurations} \label{sec:combinatorial-objects}
\subsection{Twenty-vertex configurations on quadrangular domains}
Here we define $20$V configurations.
The reader may find it helpful to compare the definition with the example shown in \cref{fig:20V-config-example}.

\begin{defi}
    We consider the triangular lattice consisting of the standard horizontal and vertical edges in $\Z^2$ plus the diagonal edges connecting $(i, j)$ and $(i+1, j-1)$ for all $i, j \in \Z$.

    Let $\mathbf{k} = (k_1, k_2, \ldots, k_n)$ be a strictly increasing sequence of positive integers.
    Define the quadrangular domain $\mathcal{Q}_{\mathbf{k}}$ to be the domain in the triangular lattice whose lattice points are in the closed quadrangle bounded by the lines
    \begin{equation*}
        \qquad j=k_n \quad (\text{north}),\qquad i+j=2 \quad (\text{south}),\qquad i=n \quad (\text{east}), \qquad i=1 \quad (\text{west}).
    \end{equation*}
    Its internal edges are the edges of the triangular lattice whose endpoints both lie in $\mathcal{Q}_{\mathbf{k}}$, and its boundary edges are those with exactly one endpoint in $\mathcal{Q}_{\mathbf{k}}$.
    For a boundary edge, incoming and outgoing are understood with respect to its endpoint in $\mathcal{Q}_{\mathbf{k}}$; the other endpoint is not regarded as a vertex of the domain.

    The boundary edges are oriented according to the following boundary conditions:
    \begin{itemize}
        \item North boundary: All boundary edges are outgoing;
        \item South boundary: All boundary edges are outgoing;
        \item East boundary: All boundary edges are incoming;
        \item West boundary: The horizontal boundary edges at positions~$k_1, \ldots, k_n$, counted from the bottom, are incoming, while the other horizontal boundary edges and all diagonal boundary edges are outgoing.
    \end{itemize}

    A \myemph{$20$V configuration} of $\mathcal{Q}_{\mathbf{k}}$ is an orientation of the internal edges which, together with the fixed orientations of the boundary edges, satisfies the \myemph{ice rule} at every lattice point of $\mathcal{Q}_{\mathbf{k}}$, that is, each such point is incident to exactly three incoming and three outgoing edges.
\end{defi}

\begin{figure}[hbt]
    \centering
    \includegraphics[scale=0.8]{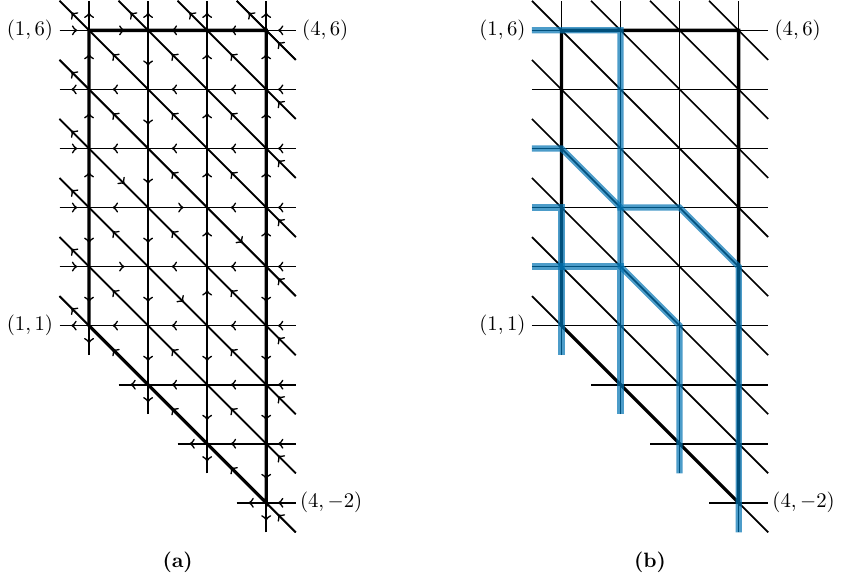}
    \caption{(a) An example of a $20$V configuration for $\mathbf{k}=(2,3,4,6)$. (b) The path representation of (a), in which the edges occupied by the paths are colored blue.} \label{fig:20V-config-example}
\end{figure}

In the special case $\mathbf{k} = (1, 2, \ldots, n)$, the domain~$\mathcal{Q}_{(1, 2, \ldots, n)}$ coincides with the quadrangular domain studied in \cite{MR4245068,MR4395233}.

Every $20$V configuration on $\mathcal{Q}_{\mathbf{k}}$ can be identified with a family of $n$ osculating Schr\"{o}der paths, where the $\ell$-th path starts at $(1,k_\ell)$ and ends at $(\ell,2-\ell)$ for $\ell = 1,2,\ldots,n$.
An edge of the domain, including a boundary edge, is used by the path family if and only if it is oriented rightward, downward, or southeastward.
See \cref{tab:conf-path-identification-20V} for the correspondence between local edge orientations and path segments around a vertex.
Throughout this paper, we represent $20$V configurations interchangeably as edge orientations and as families of osculating Schr\"{o}der paths.

\begin{table}[htb]
    \centering
    \begin{tabular}{c|c|c|c|c|c|c|c|c|c}
        \begin{tikzpicture}[scale=0.55,very thick,decoration={
        markings,
        mark=at position 0.6 with {\arrow[scale=0.8]{>}}}
        ]
            \draw[postaction={decorate}] (-1,0) -- (0,0);
            \draw[postaction={decorate}] (0,0) -- (1,0);
            \draw[postaction={decorate}] (0,1) -- (0,0);
            \draw[postaction={decorate}] (0,0) -- (0,-1);
            \draw[postaction={decorate}] (-1,1) -- (0,0);
            \draw[postaction={decorate}] (0,0) -- (1,-1);
        \end{tikzpicture}
        &
        \begin{tikzpicture}[scale=0.55,shift={(2.5,0)},
        very thick,decoration={
        markings,
        mark=at position 0.6 with {\arrow[scale=0.8]{>}}}
        ]
            \draw[postaction={decorate}] (-1,0) -- (0,0);
            \draw[postaction={decorate}] (0,0) -- (1,0);
            \draw[postaction={decorate}] (0,0) -- (0,1);
            \draw[postaction={decorate}] (0,-1) -- (0,0);
            \draw[postaction={decorate}] (-1,1) -- (0,0);
            \draw[postaction={decorate}] (0,0) -- (1,-1);
        \end{tikzpicture}
        &
        \begin{tikzpicture}[scale=0.55,shift={(5,0)},
        very thick,decoration={
        markings,
        mark=at position 0.6 with {\arrow[scale=0.8]{>}}}
        ]
            \draw[postaction={decorate}] (-1,0) -- (0,0);
            \draw[postaction={decorate}] (0,0) -- (1,0);
            \draw[postaction={decorate}] (0,0) -- (0,1);
            \draw[postaction={decorate}] (0,0) -- (0,-1);
            \draw[postaction={decorate}] (-1,1) -- (0,0);
            \draw[postaction={decorate}] (1,-1) -- (0,0);
        \end{tikzpicture}
        &
        \begin{tikzpicture}[scale=0.55,shift={(7.5,0)},
        very thick,decoration={
        markings,
        mark=at position 0.6 with {\arrow[scale=0.8]{>}}}
        ]
            \draw[postaction={decorate}] (-1,0) -- (0,0);
            \draw[postaction={decorate}] (1,0) -- (0,0);
            \draw[postaction={decorate}] (0,0) -- (0,1);
            \draw[postaction={decorate}] (0,0) -- (0,-1);
            \draw[postaction={decorate}] (-1,1) -- (0,0);
            \draw[postaction={decorate}] (0,0) -- (1,-1);
        \end{tikzpicture}
        &
        \begin{tikzpicture}[scale=0.55,shift={(10,0)},
        very thick,decoration={
        markings,
        mark=at position 0.6 with {\arrow[scale=0.8]{>}}}
        ]
            \draw[postaction={decorate}] (-1,0) -- (0,0);
            \draw[postaction={decorate}] (0,0) -- (1,0);
            \draw[postaction={decorate}] (0,1) -- (0,0);
            \draw[postaction={decorate}] (0,-1) -- (0,0);
            \draw[postaction={decorate}] (0,0) -- (-1,1);
            \draw[postaction={decorate}] (0,0) -- (1,-1);
        \end{tikzpicture}
        &
        \begin{tikzpicture}[scale=0.55,shift={(12.5,0)},
        very thick,decoration={
        markings,
        mark=at position 0.6 with {\arrow[scale=0.8]{>}}}
        ]
            \draw[postaction={decorate}] (-1,0) -- (0,0);
            \draw[postaction={decorate}] (0,0) -- (1,0);
            \draw[postaction={decorate}] (0,1) -- (0,0);
            \draw[postaction={decorate}] (0,0) -- (0,-1);
            \draw[postaction={decorate}] (0,0) -- (-1,1);
            \draw[postaction={decorate}] (1,-1) -- (0,0);
        \end{tikzpicture}
        &
        \begin{tikzpicture}[scale=0.55,
        very thick,decoration={
        markings,
        mark=at position 0.6 with {\arrow[scale=0.8]{>}}}
        ]
            \draw[postaction={decorate}] (-1,0) -- (0,0);
            \draw[postaction={decorate}] (1,0) -- (0,0);
            \draw[postaction={decorate}] (0,1) -- (0,0);
            \draw[postaction={decorate}] (0,0) -- (0,-1);
            \draw[postaction={decorate}] (0,0) -- (-1,1);
            \draw[postaction={decorate}] (0,0) -- (1,-1);
        \end{tikzpicture}
        &
        \begin{tikzpicture}[scale=0.55,
        very thick,decoration={
        markings,
        mark=at position 0.6 with {\arrow[scale=0.8]{>}}}
        ]
            \draw[postaction={decorate}] (-1,0) -- (0,0);
            \draw[postaction={decorate}] (0,0) -- (1,0);
            \draw[postaction={decorate}] (0,0) -- (0,1);
            \draw[postaction={decorate}] (0,-1) -- (0,0);
            \draw[postaction={decorate}] (0,0) -- (-1,1);
            \draw[postaction={decorate}] (1,-1) -- (0,0);
        \end{tikzpicture}
        &
        \begin{tikzpicture}[scale=0.55,
        very thick,decoration={
        markings,
        mark=at position 0.6 with {\arrow[scale=0.8]{>}}}
        ]
            \draw[postaction={decorate}] (-1,0) -- (0,0);
            \draw[postaction={decorate}] (1,0) -- (0,0);
            \draw[postaction={decorate}] (0,0) -- (0,1);
            \draw[postaction={decorate}] (0,-1) -- (0,0);
            \draw[postaction={decorate}] (0,0) -- (-1,1);
            \draw[postaction={decorate}] (0,0) -- (1,-1);
        \end{tikzpicture}
        &
        \begin{tikzpicture}[scale=0.55,
        very thick,decoration={
        markings,
        mark=at position 0.6 with {\arrow[scale=0.8]{>}}}
        ]
            \draw[postaction={decorate}] (-1,0) -- (0,0);
            \draw[postaction={decorate}] (1,0) -- (0,0);
            \draw[postaction={decorate}] (0,0) -- (0,1);
            \draw[postaction={decorate}] (0,0) -- (0,-1);
            \draw[postaction={decorate}] (0,0) -- (-1,1);
            \draw[postaction={decorate}] (1,-1) -- (0,0);
        \end{tikzpicture}
        \\
        \begin{tikzpicture}[scale=0.55]
            \draw[color=wongblue, ultra thick] (-1,0) -- (0,0);
            \draw[color=wongblue, ultra thick] (0,0) -- (1,0);
            \draw[color=wongblue, ultra thick] (0,1) -- (0,0);
            \draw[color=wongblue, ultra thick] (0,0) -- (0,-1);
            \draw[color=wongblue, ultra thick] (-1,1) -- (1,-1);
        \end{tikzpicture}
        &
        \begin{tikzpicture}[scale=0.55]
            \draw[color=wongblue, ultra thick] (-1,0) -- (0,0);
            \draw[color=wongblue, ultra thick] (0,0) -- (1,0);
            \draw (0,1) -- (0,0);
            \draw (0,0) -- (0,-1);
            \draw[color=wongblue, ultra thick] (-1,1) -- (1,-1);
        \end{tikzpicture}
        &
        \begin{tikzpicture}[scale=0.55]
            \draw[color=wongblue, ultra thick] (-1,0) -- (0,0);
            \draw[color=wongblue, ultra thick] (0,0) -- (1,0);
            \draw (0,1) -- (0,0);
            \draw[color=wongblue, ultra thick] (0,0) -- (0,-1);
            \draw[color=wongblue, ultra thick] (-1,1) -- (0,0);
            \draw (0,0) -- (1,-1);
        \end{tikzpicture}
        &
        \begin{tikzpicture}[scale=0.55]
            \draw[color=wongblue, ultra thick] (-1,0) -- (0,0);
            \draw (0,0) -- (1,0);
            \draw (0,1) -- (0,0);
            \draw[color=wongblue, ultra thick] (0,0) -- (0,-1);
            \draw[color=wongblue, ultra thick] (-1,1) -- (0,0);
            \draw[color=wongblue, ultra thick] (0,0) -- (1,-1);
        \end{tikzpicture}
        &
        \begin{tikzpicture}[scale=0.55]
            \draw[color=wongblue, ultra thick] (-1,0) -- (0,0);
            \draw[color=wongblue, ultra thick] (0,0) -- (1,0);
            \draw[color=wongblue, ultra thick] (0,1) -- (0,0);
            \draw (0,0) -- (0,-1);
            \draw (-1,1) -- (0,0);
            \draw[color=wongblue, ultra thick] (0,0) -- (1,-1);
        \end{tikzpicture}
        &
        \begin{tikzpicture}[scale=0.55]
            \draw[color=wongblue, ultra thick] (-1,0) -- (0,0);
            \draw[color=wongblue, ultra thick] (0,0) -- (1,0);
            \draw[color=wongblue, ultra thick] (0,1) -- (0,0);
            \draw[color=wongblue, ultra thick] (0,0) -- (0,-1);
            \draw (-1,1) -- (0,0);
            \draw (0,0) -- (1,-1);
        \end{tikzpicture}
        &
        \begin{tikzpicture}[scale=0.55]
            \draw[color=wongblue, ultra thick] (-1,0) -- (0,0);
            \draw (0,0) -- (1,0);
            \draw[color=wongblue, ultra thick] (0,1) -- (0,0);
            \draw[color=wongblue, ultra thick] (0,0) -- (0,-1);
            \draw (-1,1) -- (0,0);
            \draw[color=wongblue, ultra thick] (0,0) -- (1,-1);
        \end{tikzpicture}
        &
        \begin{tikzpicture}[scale=0.55]
            \draw[color=wongblue, ultra thick] (-1,0) -- (0,0);
            \draw[color=wongblue, ultra thick] (0,0) -- (1,0);
            \draw (0,1) -- (0,0);
            \draw (0,0) -- (0,-1);
            \draw (-1,1) -- (0,0);
            \draw (0,0) -- (1,-1);
        \end{tikzpicture}
        &
        \begin{tikzpicture}[scale=0.55]
            \draw[color=wongblue, ultra thick] (-1,0) -- (0,0);
            \draw (0,0) -- (1,0);
            \draw (0,1) -- (0,0);
            \draw (0,0) -- (0,-1);
            \draw (-1,1) -- (0,0);
            \draw[color=wongblue, ultra thick] (0,0) -- (1,-1);
        \end{tikzpicture}
        &
        \begin{tikzpicture}[scale=0.55]
            \draw[color=wongblue, ultra thick] (-1,0) -- (0,0);
            \draw (0,0) -- (1,0);
            \draw (0,1) -- (0,0);
            \draw[color=wongblue, ultra thick] (0,0) -- (0,-1);
            \draw (-1,1) -- (0,0);
            \draw (0,0) -- (1,-1);
        \end{tikzpicture}
        \\
        \hline
        \begin{tikzpicture}[scale=0.55,very thick,decoration={
        markings,
        mark=at position 0.4 with {\arrowreversed[scale=0.8]{>}}}
        ]
            \node at (0,1.5) {};
            \draw[postaction={decorate}] (-1,0) -- (0,0);
            \draw[postaction={decorate}] (0,0) -- (1,0);
            \draw[postaction={decorate}] (0,1) -- (0,0);
            \draw[postaction={decorate}] (0,0) -- (0,-1);
            \draw[postaction={decorate}] (-1,1) -- (0,0);
            \draw[postaction={decorate}] (0,0) -- (1,-1);
        \end{tikzpicture}
        &
        \begin{tikzpicture}[scale=0.55,shift={(2.5,0)},
        very thick,decoration={
        markings,
        mark=at position 0.4 with {\arrowreversed[scale=0.8]{>}}}
        ]
            \draw[postaction={decorate}] (-1,0) -- (0,0);
            \draw[postaction={decorate}] (0,0) -- (1,0);
            \draw[postaction={decorate}] (0,0) -- (0,1);
            \draw[postaction={decorate}] (0,-1) -- (0,0);
            \draw[postaction={decorate}] (-1,1) -- (0,0);
            \draw[postaction={decorate}] (0,0) -- (1,-1);
        \end{tikzpicture}
        &
        \begin{tikzpicture}[scale=0.55,shift={(5,0)},
        very thick,decoration={
        markings,
        mark=at position 0.4 with {\arrowreversed[scale=0.8]{>}}}
        ]
            \draw[postaction={decorate}] (-1,0) -- (0,0);
            \draw[postaction={decorate}] (0,0) -- (1,0);
            \draw[postaction={decorate}] (0,0) -- (0,1);
            \draw[postaction={decorate}] (0,0) -- (0,-1);
            \draw[postaction={decorate}] (-1,1) -- (0,0);
            \draw[postaction={decorate}] (1,-1) -- (0,0);
        \end{tikzpicture}
        &
        \begin{tikzpicture}[scale=0.55,shift={(7.5,0)},
        very thick,decoration={
        markings,
        mark=at position 0.4 with {\arrowreversed[scale=0.8]{>}}}
        ]
            \draw[postaction={decorate}] (-1,0) -- (0,0);
            \draw[postaction={decorate}] (1,0) -- (0,0);
            \draw[postaction={decorate}] (0,0) -- (0,1);
            \draw[postaction={decorate}] (0,0) -- (0,-1);
            \draw[postaction={decorate}] (-1,1) -- (0,0);
            \draw[postaction={decorate}] (0,0) -- (1,-1);
        \end{tikzpicture}
        &
        \begin{tikzpicture}[scale=0.55,shift={(10,0)},
        very thick,decoration={
        markings,
        mark=at position 0.4 with {\arrowreversed[scale=0.8]{>}}}
        ]
            \draw[postaction={decorate}] (-1,0) -- (0,0);
            \draw[postaction={decorate}] (0,0) -- (1,0);
            \draw[postaction={decorate}] (0,1) -- (0,0);
            \draw[postaction={decorate}] (0,-1) -- (0,0);
            \draw[postaction={decorate}] (0,0) -- (-1,1);
            \draw[postaction={decorate}] (0,0) -- (1,-1);
        \end{tikzpicture}
        &
        \begin{tikzpicture}[scale=0.55,shift={(12.5,0)},
        very thick,decoration={
        markings,
        mark=at position 0.4 with {\arrowreversed[scale=0.8]{>}}}
        ]
            \draw[postaction={decorate}] (-1,0) -- (0,0);
            \draw[postaction={decorate}] (0,0) -- (1,0);
            \draw[postaction={decorate}] (0,1) -- (0,0);
            \draw[postaction={decorate}] (0,0) -- (0,-1);
            \draw[postaction={decorate}] (0,0) -- (-1,1);
            \draw[postaction={decorate}] (1,-1) -- (0,0);
        \end{tikzpicture}
        &
        \begin{tikzpicture}[scale=0.55,
        very thick,decoration={
        markings,
        mark=at position 0.4 with {\arrowreversed[scale=0.8]{>}}}
        ]
            \draw[postaction={decorate}] (-1,0) -- (0,0);
            \draw[postaction={decorate}] (1,0) -- (0,0);
            \draw[postaction={decorate}] (0,1) -- (0,0);
            \draw[postaction={decorate}] (0,0) -- (0,-1);
            \draw[postaction={decorate}] (0,0) -- (-1,1);
            \draw[postaction={decorate}] (0,0) -- (1,-1);
        \end{tikzpicture}
        &
        \begin{tikzpicture}[scale=0.55,
        very thick,decoration={
        markings,
        mark=at position 0.4 with {\arrowreversed[scale=0.8]{>}}}
        ]
            \draw[postaction={decorate}] (-1,0) -- (0,0);
            \draw[postaction={decorate}] (0,0) -- (1,0);
            \draw[postaction={decorate}] (0,0) -- (0,1);
            \draw[postaction={decorate}] (0,-1) -- (0,0);
            \draw[postaction={decorate}] (0,0) -- (-1,1);
            \draw[postaction={decorate}] (1,-1) -- (0,0);
        \end{tikzpicture}
        &
        \begin{tikzpicture}[scale=0.55,
        very thick,decoration={
        markings,
        mark=at position 0.4 with {\arrowreversed[scale=0.8]{>}}}
        ]
            \draw[postaction={decorate}] (-1,0) -- (0,0);
            \draw[postaction={decorate}] (1,0) -- (0,0);
            \draw[postaction={decorate}] (0,0) -- (0,1);
            \draw[postaction={decorate}] (0,-1) -- (0,0);
            \draw[postaction={decorate}] (0,0) -- (-1,1);
            \draw[postaction={decorate}] (0,0) -- (1,-1);
        \end{tikzpicture}
        &
        \begin{tikzpicture}[scale=0.55,
        very thick,decoration={
        markings,
        mark=at position 0.4 with {\arrowreversed[scale=0.8]{>}}}
        ]
            \draw[postaction={decorate}] (-1,0) -- (0,0);
            \draw[postaction={decorate}] (1,0) -- (0,0);
            \draw[postaction={decorate}] (0,0) -- (0,1);
            \draw[postaction={decorate}] (0,0) -- (0,-1);
            \draw[postaction={decorate}] (0,0) -- (-1,1);
            \draw[postaction={decorate}] (1,-1) -- (0,0);
        \end{tikzpicture}
        \\
        \begin{tikzpicture}[scale=0.55]
            \draw (-1,0) -- (0,0);
            \draw (0,0) -- (1,0);
            \draw (0,1) -- (0,0);
            \draw (0,0) -- (0,-1);
            \draw (-1,1) -- (1,-1);
        \end{tikzpicture}
        &
        \begin{tikzpicture}[scale=0.55]
            \draw (-1,0) -- (0,0);
            \draw (0,0) -- (1,0);
            \draw[color=wongblue, ultra thick] (0,1) -- (0,0);
            \draw[color=wongblue, ultra thick] (0,0) -- (0,-1);
            \draw (-1,1) -- (1,-1);
        \end{tikzpicture}
        &
        \begin{tikzpicture}[scale=0.55]
            \draw (-1,0) -- (0,0);
            \draw (0,0) -- (1,0);
            \draw[color=wongblue, ultra thick] (0,1) -- (0,0);
            \draw (0,0) -- (0,-1);
            \draw (-1,1) -- (0,0);
            \draw[color=wongblue, ultra thick] (0,0) -- (1,-1);
        \end{tikzpicture}
        &
        \begin{tikzpicture}[scale=0.55]
            \draw (-1,0) -- (0,0);
            \draw[color=wongblue, ultra thick] (0,0) -- (1,0);
            \draw[color=wongblue, ultra thick] (0,1) -- (0,0);
            \draw (0,0) -- (0,-1);
            \draw (-1,1) -- (0,0);
            \draw (0,0) -- (1,-1);
        \end{tikzpicture}
        &
        \begin{tikzpicture}[scale=0.55]
            \draw (-1,0) -- (0,0);
            \draw (0,0) -- (1,0);
            \draw (0,1) -- (0,0);
            \draw[color=wongblue, ultra thick] (0,0) -- (0,-1);
            \draw[color=wongblue, ultra thick] (-1,1) -- (0,0);
            \draw (0,0) -- (1,-1);
        \end{tikzpicture}
        &
        \begin{tikzpicture}[scale=0.55]
            \draw (-1,0) -- (0,0);
            \draw (0,0) -- (1,0);
            \draw (0,1) -- (0,0);
            \draw (0,0) -- (0,-1);
            \draw[color=wongblue, ultra thick] (-1,1) -- (0,0);
            \draw[color=wongblue, ultra thick] (0,0) -- (1,-1);
        \end{tikzpicture}
        &
        \begin{tikzpicture}[scale=0.55]
            \draw (-1,0) -- (0,0);
            \draw[color=wongblue, ultra thick] (0,0) -- (1,0);
            \draw (0,1) -- (0,0);
            \draw (0,0) -- (0,-1);
            \draw[color=wongblue, ultra thick] (-1,1) -- (0,0);
            \draw (0,0) -- (1,-1);
        \end{tikzpicture}
        &
        \begin{tikzpicture}[scale=0.55]
            \draw (-1,0) -- (0,0);
            \draw (0,0) -- (1,0);
            \draw[color=wongblue, ultra thick] (0,1) -- (0,0);
            \draw[color=wongblue, ultra thick] (0,0) -- (0,-1);
            \draw[color=wongblue, ultra thick] (-1,1) -- (0,0);
            \draw[color=wongblue, ultra thick] (0,0) -- (1,-1);
        \end{tikzpicture}
        &
        \begin{tikzpicture}[scale=0.55]
            \draw (-1,0) -- (0,0);
            \draw[color=wongblue, ultra thick] (0,0) -- (1,0);
            \draw[color=wongblue, ultra thick] (0,1) -- (0,0);
            \draw[color=wongblue, ultra thick] (0,0) -- (0,-1);
            \draw[color=wongblue, ultra thick] (-1,1) -- (0,0);
            \draw (0,0) -- (1,-1);
        \end{tikzpicture}
        &
        \begin{tikzpicture}[scale=0.55]
            \draw (-1,0) -- (0,0);
            \draw[color=wongblue, ultra thick] (0,0) -- (1,0);
            \draw[color=wongblue, ultra thick] (0,1) -- (0,0);
            \draw (0,0) -- (0,-1);
            \draw[color=wongblue, ultra thick] (-1,1) -- (0,0);
            \draw[color=wongblue, ultra thick] (0,0) -- (1,-1);
        \end{tikzpicture}
    \end{tabular}
    \caption{The identification of $20$V configurations and families of osculating Schr\"{o}der paths.} \label{tab:conf-path-identification-20V}
\end{table}

\subsection{Mixed six-vertex configurations}
Here we define mixed $6$V configurations.
The reader may find it helpful to compare the definition with the example shown in \cref{fig:example-m6V-config}.

\begin{defi}
    Consider the square lattice~$\Z^2$ consisting of the standard horizontal and vertical edges.

    Assume $\mathbf{k} = (k_1, k_2, \ldots, k_n)$ is a strictly increasing sequence of positive integers.
    Define the rectangular domain $\mathcal{M}_{\mathbf{k}}$ to be the domain in the square lattice whose lattice points are in the closed quadrangle bounded by the lines
    \begin{equation*}
        j = 2k_n-1 \quad (\text{north}), \qquad j = 1 \quad (\text{south}), \qquad i = n \quad (\text{east}), \qquad i = 1 \quad (\text{west}).
    \end{equation*}
    Its internal edges are the edges of the square lattice whose endpoints both lie in $\mathcal{M}_{\mathbf{k}}$, and its boundary edges are those with exactly one endpoint in $\mathcal{M}_{\mathbf{k}}$.
    For a boundary edge, incoming and outgoing are understood with respect to its endpoint in $\mathcal{M}_{\mathbf{k}}$; the other endpoint is not regarded as a vertex of the domain.

    The boundary edges are oriented as follows:
    \begin{itemize}
        \item North boundary: All boundary edges are outgoing;
        \item South boundary: All boundary edges are outgoing;
        \item East boundary: All boundary edges are incoming;
        \item West boundary: The horizontal boundary edges at positions~$2k_1-1, \ldots, 2k_n-1$, counted from the bottom, are incoming, while all other horizontal boundary edges are outgoing. (Note that incoming boundary edges can only occur at odd positions.)
    \end{itemize}

    A \myemph{mixed $6$V configuration} of $\mathcal{M}_{\mathbf{k}}$ is an orientation of the internal edges which, together with the fixed orientations of the boundary edges, satisfies the \myemph{ice rule} at every lattice point of $\mathcal{M}_{\mathbf{k}}$, that is, each such point is incident to exactly two incoming and two outgoing edges.
\end{defi}

\begin{figure}[hbt]
    \centering
    \includegraphics[scale=0.6]{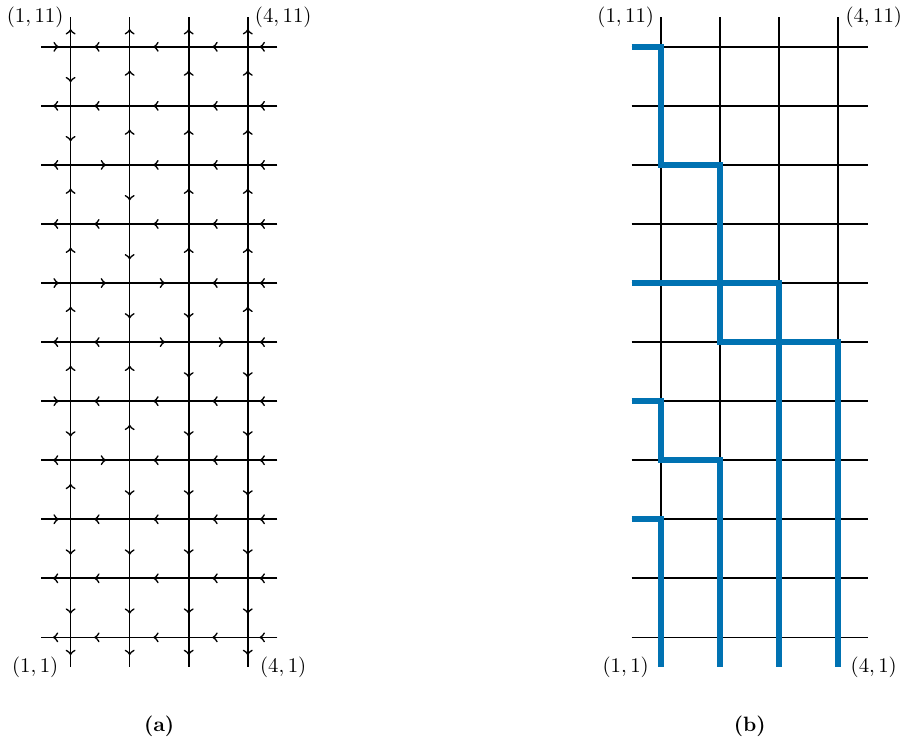}
    \caption{(a) A mixed $6$V configuration on the domain $\mathcal{M}_{(2, 3, 4, 6)}$. (b) The path representation of (a).}
    \label{fig:example-m6V-config}
\end{figure}

In the special case $\mathbf{k} = (1, 2, \ldots, n)$, the domain~$\mathcal{M}_{(1, 2, \ldots, n)}$ coincides with the rectangular domain studied in \cite{MR4395233}.

Every mixed $6$V configuration can be identified with a family of $n$ osculating lattice paths, where the $\ell$-th path starts at $(1,2k_\ell-1)$ and ends at $(\ell,1)$ for $\ell = 1,2,\ldots,n$.
An edge of the domain, including a boundary edge, is used by the path family if and only if it is oriented rightward or downward.
See \cref{tab:conf-path-identification} for the correspondence between local edge orientations and path segments around a vertex.
\begin{table}[htb]
    \centering
    \begin{tabular}{c|c|c|c|c|c}
        \begin{tikzpicture}[scale=0.7,very thick,decoration={
        markings,
        mark=at position 0.6 with {\arrow{>}}}
        ]
            \draw[postaction={decorate}] (-1,0) -- (0,0);
            \draw[postaction={decorate}] (0,0) -- (1,0);
            \draw[postaction={decorate}] (0,1) -- (0,0);
            \draw[postaction={decorate}] (0,0) -- (0,-1);
        \end{tikzpicture}
        &
        \begin{tikzpicture}[scale=0.7,shift={(2.5,0)},
        very thick,decoration={
        markings,
        mark=at position 0.6 with {\arrow{>}}}
        ]
            \draw[postaction={decorate}] (0,0) -- (-1,0);
            \draw[postaction={decorate}] (1,0) -- (0,0);
            \draw[postaction={decorate}] (0,0) -- (0,1);
            \draw[postaction={decorate}] (0,-1) -- (0,0);
        \end{tikzpicture}
        &
        \begin{tikzpicture}[scale=0.7,shift={(5,0)},
        very thick,decoration={
        markings,
        mark=at position 0.6 with {\arrow{>}}}
        ]
            \draw[postaction={decorate}] (-1,0) -- (0,0);
            \draw[postaction={decorate}] (0,0) -- (1,0);
            \draw[postaction={decorate}] (0,0) -- (0,1);
            \draw[postaction={decorate}] (0,-1) -- (0,0);
        \end{tikzpicture}
        &
        \begin{tikzpicture}[scale=0.7,shift={(7.5,0)},
        very thick,decoration={
        markings,
        mark=at position 0.6 with {\arrow{>}}}
        ]
            \draw[postaction={decorate}] (0,0) -- (-1,0);
            \draw[postaction={decorate}] (1,0) -- (0,0);
            \draw[postaction={decorate}] (0,1) -- (0,0);
            \draw[postaction={decorate}] (0,0) -- (0,-1);
        \end{tikzpicture}
        &
        \begin{tikzpicture}[scale=0.7,shift={(10,0)},
        very thick,decoration={
        markings,
        mark=at position 0.6 with {\arrow{>}}}
        ]
            \draw[postaction={decorate}] (-1,0) -- (0,0);
            \draw[postaction={decorate}] (1,0) -- (0,0);
            \draw[postaction={decorate}] (0,0) -- (0,1);
            \draw[postaction={decorate}] (0,0) -- (0,-1);
        \end{tikzpicture}
        &
        \begin{tikzpicture}[scale=0.7,shift={(12.5,0)},
        very thick,decoration={
        markings,
        mark=at position 0.6 with {\arrow{>}}}
        ]
            \draw[postaction={decorate}] (0,0) -- (-1,0);
            \draw[postaction={decorate}] (0,0) -- (1,0);
            \draw[postaction={decorate}] (0,1) -- (0,0);
            \draw[postaction={decorate}] (0,-1) -- (0,0);
        \end{tikzpicture}
        \\
        \\
        \begin{tikzpicture}
            \draw[color=wongblue, ultra thick] (-1,0) -- (0,0);
            \draw[color=wongblue, ultra thick] (0,0) -- (1,0);
            \draw[color=wongblue, ultra thick] (0,1) -- (0,0);
            \draw[color=wongblue, ultra thick] (0,0) -- (0,-1);
        \end{tikzpicture}
        &
        \begin{tikzpicture}
            \draw (-1,0) -- (0,0);
            \draw (0,0) -- (1,0);
            \draw (0,1) -- (0,0);
            \draw (0,0) -- (0,-1);
        \end{tikzpicture}
        &
        \begin{tikzpicture}
            \draw[color=wongblue, ultra thick] (-1,0) -- (0,0);
            \draw[color=wongblue, ultra thick] (0,0) -- (1,0);
            \draw (0,1) -- (0,0);
            \draw (0,0) -- (0,-1);
        \end{tikzpicture}
        &
        \begin{tikzpicture}
            \draw (-1,0) -- (0,0);
            \draw (0,0) -- (1,0);
            \draw[color=wongblue, ultra thick] (0,1) -- (0,0);
            \draw[color=wongblue, ultra thick] (0,0) -- (0,-1);
        \end{tikzpicture}
        &
        \begin{tikzpicture}
            \draw[color=wongblue, ultra thick] (-1,0) -- (0,0);
            \draw (0,0) -- (1,0);
            \draw (0,1) -- (0,0);
            \draw[color=wongblue, ultra thick] (0,0) -- (0,-1);
        \end{tikzpicture}
        &
        \begin{tikzpicture}
            \draw (-1,0) -- (0,0);
            \draw[color=wongblue, ultra thick] (0,0) -- (1,0);
            \draw[color=wongblue, ultra thick] (0,1) -- (0,0);
            \draw (0,0) -- (0,-1);
        \end{tikzpicture}
    \end{tabular}
    \caption{The identification of mixed $6$V configurations and families of osculating lattice paths.} \label{tab:conf-path-identification}
\end{table}
Similar to $20$V configurations, throughout this paper, we represent mixed $6$V configurations interchangeably as edge orientations and as families of osculating lattice paths.

\section{Probabilistic bijections and Yang--Baxter moves} \label{sec:local-prob-map}
In this section, we first define probabilistic bijections (\cref{defi:probabilistic-bijections}).
We then interpret Yang--Baxter moves as probabilistic bijections (\cref{prop:Yang--Baxter-moves-as-probabilistic-bijections}).
By composing these probabilistic bijections, we construct a probabilistic bijection between $20$V configurations and mixed $6$V configurations (\cref{prop:twenty-mixed6V-prob-bij}).

\subsection{Probabilistic bijections} \label{subsec:prob-bij}
The following definition comes from \cite{MR4031106}, where the authors use the term ``bijectivization'' instead of ``probabilistic bijection''.
\begin{defi} \label{defi:probabilistic-bijections}
    Let $X$ and $Y$ be two finite sets with weight functions $\omega_X \colon X \to A$ and $\omega_Y \colon Y \to A$, where $A$ is a possibly noncommutative unital algebra.
    A \myemph{probabilistic bijection} from $(X, \omega_X)$ to $(Y, \omega_Y)$ is a pair of maps~$\mathcal{P}, \overline{\mathcal{P}} \colon X \times Y \to A$ such that
    \begin{enumerate}
        \item for each $x \in X$, we have $\sum_{y \in Y} \mathcal{P}(x, y) = 1$;
        \item for each $y \in Y$, we have $\sum_{x \in X} \overline{\mathcal{P}}(x, y) = 1$;
        \item for each $x \in X$ and $y \in Y$, we have $\omega_X(x) \mathcal{P}(x, y) = \overline{\mathcal{P}}(x, y) \omega_Y(y)$.
    \end{enumerate}
\end{defi}
Probabilistic bijections generalize bijections.
To see that, now suppose that $f\colon X \to Y$ is a bijection.
Define the weight functions~$\omega_X\colon X \to \R$ and $\omega_Y\colon Y \to \R$ by $\omega_X \equiv 1$ and $\omega_Y \equiv 1$.
Then, the maps~$\mathcal{P}, \overline{\mathcal{P}}\colon  X \times Y \to \R$ given by $\mathcal{P}(x, y)=\overline{\mathcal{P}}(x,y)=[y = f(x)]$ form a probabilistic bijection between $X$ and $Y$.
(This is a special case of \cref{lemm:proper-weight-distribution-is-prob-bij} given below.)

\begin{rema}
    Let $(\mathcal{P}, \overline{\mathcal{P}})$ be a probabilistic bijection from $(X,\omega_X)$ to $(Y,\omega_Y)$.
    Suppose that $\omega_X(x) \mathcal{P}(x, y)= \overline{\mathcal{P}}(x, y) \omega_Y(y) \in \R_{\geq 0}$ for all $(x, y) \in X \times Y$, and that $\sum_{x \in X} \omega_X(x) = \sum_{y \in Y} \omega_Y(y) = 1$.
    The latter condition can always be achieved, provided the total weights are nonzero, by normalizing each weight function by its total weight.
    Then, $\mathbb{P}(x, y) \coloneqq \omega_X(x) \mathcal{P}(x, y) = \overline{\mathcal{P}}(x, y) \omega_Y(y)$ defines a joint distribution on $X \times Y$.
    In this case, $\omega_X$ and $\omega_Y$ are the marginal distributions of $X$ and $Y$, respectively, while $\mathcal{P}$ and $\overline{\mathcal{P}}$ are the corresponding conditional distributions.
    See also \cite[2.2]{MR4031106} and \cite[Remark 4.1.4]{MR4311960} for probabilistic interpretations of probabilistic bijections.
\end{rema}

An important property of probabilistic bijections is that they preserve the weighted enumerations as follows.
\begin{lemm} \label{lemm:weight-preservation}
    Let $X$ and $Y$ be two finite sets with weight functions $\omega_X \colon X \to A, \omega_Y \colon Y \to A$, and assume that there exists a probabilistic bijection $(\mathcal{P}, \overline{\mathcal{P}})$ between $X$ and $Y$.
    Then, we have
    \begin{equation*}
         \sum_{x \in X} \omega_X(x) = \sum_{y \in Y} \omega_Y(y).
    \end{equation*}
\end{lemm}
\begin{proof}
    We use the first, third, and second conditions of \cref{defi:probabilistic-bijections} successively:
    \begin{equation*}
        \sum_{x \in X} \omega_X(x) = \sum_{x \in X} \sum_{y \in Y} \omega_X(x) \mathcal{P}(x, y)
        = \sum_{y \in Y} \sum_{x \in X} \overline{\mathcal{P}}(x, y) \omega_Y(y)
        = \sum_{y \in Y} \omega_Y(y).
    \end{equation*}
\end{proof}

In \cref{sec:map-mixed-6V-to-amtri}, we will use the next lemma to show that a surjective map from mixed $6$V configurations to triple-free GT patterns gives a probabilistic bijection between them.
\begin{lemm} \label{lemm:proper-weight-distribution-is-prob-bij}
    Let $X, Y$ be two finite sets with weight functions $\omega_X \colon X \to A, \omega_Y \colon Y \to A \setminus \{0\}$, where $A$ is a field.
    Assume that there exists a map~$f \colon X \to Y$ such that
    \begin{equation*}
        \sum_{x \in f^{-1}(y)} \omega_X(x) = \omega_Y(y)
    \end{equation*}
    for each $y \in Y$.
    Then, the pair of maps~$\mathcal{P}, \overline{\mathcal{P}} \colon X \times Y \to A$ defined by
    \begin{equation*}
        \mathcal{P}(x, y) = [f(x) = y] \quad \text{and} \quad \overline{\mathcal{P}}(x, y) = \frac{\omega_X(x)}{\omega_Y(y)} [f(x) = y],
    \end{equation*}
    where $[P]$ is the Iverson bracket that equals $1$ if the proposition $P$ is true and $0$ otherwise, forms a probabilistic bijection from $(X, \omega_X)$ to $(Y, \omega_Y)$.
\end{lemm}
\begin{proof}
    We check the three conditions in the definition of probabilistic bijections.
    The first condition clearly holds since for each $x \in X$, there is only one $y \in Y$ such that $f(x) = y$.
    The second condition holds since for each $y \in Y$,
    \begin{equation*}
        \sum_{x \in X} \overline{\mathcal{P}}(x, y) = \sum_{x \in f^{-1}(y)} \frac{\omega_X(x)}{\omega_Y(y)} = \frac{1}{\omega_Y(y)} \times \sum_{x \in f^{-1}(y)} \omega_X(x) = 1.
    \end{equation*}
    The third condition holds since for each $x \in X$ and $y \in Y$, we have
    \begin{equation*}
        \omega_X(x) \mathcal{P}(x, y) = \omega_X(x) [f(x) = y] = \omega_Y(y) \cdot \frac{\omega_X(x)}{\omega_Y(y)} [f(x) = y] = \omega_Y(y) \overline{\mathcal{P}}(x, y).
    \end{equation*}
\end{proof}

The following lemma will be used to compose multiple probabilistic bijections.
\begin{lemm} \label{lemma:probabilistic-bijection-composition}
    Let $(\mathcal{P}, \overline{\mathcal{P}})$ be a probabilistic bijection from $(X, \omega_X)$ to $(Y, \omega_Y)$, and let $(\mathcal{Q}, \bar{\mathcal{Q}})$ be a probabilistic bijection from $(Y, \omega_Y)$ to $(Z, \omega_Z)$.
    Then their composition~$(\mathcal{R}, \bar{\mathcal{R}})$ defined by
    \begin{equation}
        \mathcal{R}(x, z) = \sum_{y \in Y} \mathcal{P}(x, y) \mathcal{Q}(y, z) \quad\text{and}\quad \bar{\mathcal{R}}(x, z) = \sum_{y \in Y} \bar{\mathcal{Q}}(y, z) \overline{\mathcal{P}}(x, y),
    \end{equation}
    is a probabilistic bijection from $(X, \omega_X)$ to $(Z, \omega_Z)$.
\end{lemm}
\begin{proof}
    The lemma is straightforward to prove. We only verify the third condition in \cref{defi:probabilistic-bijections}:
    \[
    \begin{aligned}
    \omega_X(x)\mathcal{R}(x,z)
    &= \sum_{y\in Y}\omega_X(x)\mathcal{P}(x,y)\mathcal{Q}(y,z)
    = \sum_{y\in Y}\overline{\mathcal{P}}(x,y)\omega_Y(y)\mathcal{Q}(y,z) \\
    &= \sum_{y\in Y}\overline{\mathcal{P}}(x,y)\overline{\mathcal{Q}}(y,z)\omega_Z(z)
    = \overline{\mathcal{R}}(x,z)\omega_Z(z).
    \end{aligned}
    \]
\end{proof}
This proof explains the order of the factors in the third condition of \cref{defi:probabilistic-bijections}: when $A$ is noncommutative, the argument requires $\omega_X(x)\mathcal{P}(x,y)=\overline{\mathcal{P}}(x,y)\omega_Y(y)$, rather than $\omega_X(x)\mathcal{P}(x,y)=\omega_Y(y)\overline{\mathcal{P}}(x,y)$.

\subsection{Yang--Baxter moves as probabilistic bijections} \label{subsec:prob-bij-Yang--Baxter}
In this subsection, we show that each Yang--Baxter move induces a probabilistic bijection between the sets of admissible configurations before and after the move.

First, we review the Yang--Baxter equation \cite{MR998375} in a form suitable for our purposes.
Namely, we focus on graphs $G$ that can be obtained from a domain~$\mathcal{Q}_\mathbf{k}$ by repeatedly applying the following local transformations of graphs:
\begin{description}
    \centering
    \item[Local bend moves] \(\quad
    \tikz[scale=0.5,baseline=-0.5ex]{
        \draw (-1, 0) -- (0,0);
        \draw (0, 0) -- (1,0);
        \draw (0, -1) -- (0,0);
        \draw (0, 0) -- (0,1);
        \draw (-1, 1) -- (0,0);
        \draw (0, 0) -- (1,-1);
    } \mathrel{\longrightarrow} \tikz[scale=0.5,baseline=-0.5ex]{
        \draw (-1, 0) -- (1,0);
        \draw (0, -1) -- (0,1);
        \draw[rounded corners] (-1, 1) -- (-0.5,0.5) -- (-0.5,-0.5) -- (0.5,-0.5) -- (1,-1);
    } \qquad\qquad \tikz[scale=0.5,baseline=-0.5ex]{
        \draw (-1, 0) -- (0,0);
        \draw (0, 0) -- (1,0);
        \draw (0, -1) -- (0,0);
        \draw (0, 0) -- (0,1);
        \draw (-1, 1) -- (0,0);
        \draw (0, 0) -- (1,-1);
    }  \mathrel{\longrightarrow} \tikz[scale=0.5,baseline=-0.5ex]{
        \draw (-1, 0) -- (1, 0);
        \draw (0, -1) -- (0, 1);
        \draw[rounded corners] (-1, 1) -- (-0.5,0.5) -- (0.5,0.5) -- (0.5,-0.5) -- (1,-1);
    }
    \)
    \item[Local flip moves] \(\quad
    \tikz[scale=0.5,baseline=-0.5ex]{
        \draw (-1, 0) -- (1, 0);
        \draw (0, -1) -- (0, 1);
        \draw[rounded corners] (-1, 1) -- (-0.5,0.5) -- (0.5,0.5) -- (0.5,-0.5) -- (1,-1);
    }  \mathrel{\longrightarrow} \tikz[scale=0.5,baseline=-0.5ex]{
        \draw (-1, 0) -- (1,0);
        \draw (0, -1) -- (0,1);
        \draw[rounded corners] (-1, 1) -- (-0.5,0.5) -- (-0.5,-0.5) -- (0.5,-0.5) -- (1,-1);
    } \qquad\qquad \tikz[scale=0.5,baseline=-0.5ex]{
        \draw (-1, 0) -- (1,0);
        \draw (0, -1) -- (0,1);
        \draw[rounded corners] (-1, 1) -- (-0.5,0.5) -- (-0.5,-0.5) -- (0.5,-0.5) -- (1,-1);
    }  \mathrel{\longrightarrow} \tikz[scale=0.5,baseline=-0.5ex]{
        \draw (-1, 0) -- (1, 0);
        \draw (0, -1) -- (0, 1);
        \draw[rounded corners] (-1, 1) -- (-0.5,0.5) -- (0.5,0.5) -- (0.5,-0.5) -- (1,-1);
    }.
    \)
\end{description}
In this paper, we refer to these transformations as \myemph{Yang--Baxter moves}.
Note that any graph $G$ obtained from $\mathcal{Q}_\mathbf{k}$ by such moves has the property that every vertex has degree four or six.
Moreover, we focus on a specific choice of vertex weights, namely those listed in \cref{tab:local-weights}, which will be explained shortly.
For further details on the weighting of $20$V configurations, we refer the reader to \cite{MR4395233,MR4245068}, where more general weights of $20$V configurations are used.

Just as boundary conditions are specified for the domain~$\mathcal{Q}_\mathbf{k}$, in general, our graph $G$ may contain some edges with prescribed orientations.
We call an orientation $x$ of the edges whose direction is not prescribed an \myemph{admissible configuration} of $G$ if the ice rule holds at every internal vertex; that is, the number of incoming edges equals the number of outgoing edges.

Each admissible configuration $x$ is assigned the \myemph{weight} $\omega_G(x)$, defined multiplicatively as the product of the \myemph{vertex weights} $\omega_v(x)$, that is,
\begin{equation*}
    \omega_{G}(x) = \prod_{v} \omega_v(x).
\end{equation*}
Here, the product is over all vertices of $G$; as in the definition of $\mathcal{Q}_{\mathbf{k}}$, the outer endpoints of the boundary edges are not regarded as vertices of $G$.

The vertex weight $\omega_v(x)$ is computed as follows.
If the degree of the vertex is four, then the weight is given according to \cref{tab:local-weights}.
If the degree of the vertex is six, the weight is defined to be the sum of the weights of all admissible configurations on the local graph obtained by applying a local bend Yang--Baxter move to the degree-$6$ vertex, as explained below.

For example, consider the following local configuration of a vertex~$v$ of degree $6$:
\tikz[scale=0.5]{
    \draw[color=wongblue,ultra thick] (-1, 0) -- (0,0);
    \draw (0, 0) -- (1,0);
    \draw[color=wongblue,ultra thick] (0, -1) -- (0,0);
    \draw (0, 0) -- (0,1);
    \draw (-1, 1) -- (0,0);
    \draw (0, 0) -- (1,-1);
}.
If the diagonal line is bent to the northeast, there is only one admissible configuration: \tikz[scale=0.5]{
    \draw (-1, 0) -- (0,0);
    \draw (0, 0) -- (1,0);
    \draw (0, -1) -- (0,0);
    \draw (0, 0) -- (0,1);
    \draw[rounded corners] (-1, 1) -- (-0.5,0.5) -- (0.5,0.5) -- (0.5,-0.5) -- (1,-1);
    \draw[color=wongblue,ultra thick] (-1,0) -- (0, 0) -- (0, -1);
},
and the weight of this configuration is $c_1 a_2 a_3$.
If the diagonal line is bent to the southwest, there are two admissible local configurations: \tikz[scale=0.5]{
    \draw (-1, 0) -- (1,0);
    \draw (0, -1) -- (0,1);
    \draw[rounded corners] (-1, 1) -- (-0.5,0.5) -- (-0.5,-0.5) -- (0.5,-0.5) -- (1,-1);
    \draw[color=wongblue,ultra thick] (-1,0) -- (-0.5,0);
    \draw[rounded corners, color=wongblue,ultra thick] (-0.5,0)--(-0.5,-0.5)--(0,-0.5);
    \draw[color=wongblue,ultra thick] (0,-0.5) -- (0,-1);
} and \tikz[scale=0.5]{
    \draw (-1, 0) -- (1,0);
    \draw (0, -1) -- (0,1);
    \draw[rounded corners] (-1, 1) -- (-0.5,0.5) -- (-0.5,-0.5) -- (0.5,-0.5) -- (1,-1);
    \draw[color=wongblue,ultra thick] (-1,0) -- (0, 0) -- (0, -1);
}, and the weights of these two admissible configurations are $a_1 c_2 c_3$ and $c_1 b_2 b_3$, respectively.
Hence, in this case, the sum of the weights of all admissible configurations is $a_1 c_2 c_3 + c_1 b_2 b_3$.
A direct calculation shows that $c_1 a_2 a_3 = a_1 c_2 c_3 + c_1 b_2 b_3 = 1$.
Thus, the weight~$\omega_v(x)$ does not depend on how the degree-$6$ vertex is resolved, meaning that it is well defined.

It is also straightforward to check that, for each of the twenty possible local configurations at a degree-$6$ vertex, the corresponding weight is independent of the choice of resolution.
Equivalently, with the vertex weights given in \cref{tab:local-weights}, the two local graphs \tikz[scale=0.5]{
    \draw (-1, 0) -- (1,0);
    \draw (0, -1) -- (0,1);
    \draw[rounded corners] (-1, 1) -- (-0.5,0.5) -- (-0.5,-0.5) -- (0.5,-0.5) -- (1,-1);
} and \tikz[scale=0.5]{
    \draw (-1, 0) -- (1, 0);
    \draw (0, -1) -- (0, 1);
    \draw[rounded corners] (-1, 1) -- (-0.5,0.5) -- (0.5,0.5) -- (0.5,-0.5) -- (1,-1);
} have the same weighted enumeration for each fixed orientation of their boundary edges.
In general, this is equivalent to saying that the vertex weights satisfy the \myemph{Yang--Baxter equation}.

Moreover, with our choice of vertex weights, all of the twenty configurations at a degree-$6$ vertex have weight~$1$.
Since $\omega_G(x)$ is defined multiplicatively and every vertex of $\mathcal{Q}_{\mathbf{k}}$ has degree~$6$, every $20$V configuration~$x$ on $\mathcal{Q}_\mathbf{k}$ satisfies $\omega_{\mathcal{Q}_{\mathbf{k}}}(x)=1$.
Thus, the weighted enumeration of $20$V configurations on $\mathcal{Q}_{\mathbf{k}}$ coincides with their ordinary enumeration.
Combining this with the local invariance of weighted enumerations by Yang--Baxter moves above, we conclude that the weighted enumeration of admissible configurations on any graph obtained from $\mathcal{Q}_{\mathbf{k}}$ by repeated applications of Yang--Baxter moves is equal to the number of $20$V configurations on $\mathcal{Q}_{\mathbf{k}}$.

\begin{table}[hbt]
    \centering
    \begin{minipage}{0.60\textwidth}
        \centering
    \begin{tabular}{c|c|c|c|c|c}
        \multicolumn{6}{c}{\textbf{H-V vertices}}\\
        \begin{tikzpicture}[scale=0.5]
            \draw[color=wongblue,ultra thick] (-1, 0) -- (0,0);
            \draw[color=wongblue,ultra thick] (0, 0) -- (1,0);
            \draw[color=wongblue,ultra thick] (0, -1) -- (0,0);
            \draw[color=wongblue,ultra thick] (0, 0) -- (0,1);
        \end{tikzpicture}
        &
        \begin{tikzpicture}[scale=0.5]
            \draw (-1, 0) -- (0,0);
            \draw (0, 0) -- (1,0);
            \draw (0, -1) -- (0,0);
            \draw (0, 0) -- (0,1);
        \end{tikzpicture}
        &
        \begin{tikzpicture}[scale=0.5]
            \draw[color=wongblue,ultra thick] (-1, 0) -- (0,0);
            \draw[color=wongblue,ultra thick] (0, 0) -- (1,0);
            \draw (0, -1) -- (0,0);
            \draw (0, 0) -- (0,1);
        \end{tikzpicture}
        &
        \begin{tikzpicture}[scale=0.5]
            \draw (-1, 0) -- (0,0);
            \draw (0, 0) -- (1,0);
            \draw[color=wongblue,ultra thick] (0, -1) -- (0,0);
            \draw[color=wongblue,ultra thick] (0, 0) -- (0,1);
        \end{tikzpicture}
        &
        \begin{tikzpicture}[scale=0.5]
            \draw[color=wongblue,ultra thick] (-1, 0) -- (0,0);
            \draw (0, 0) -- (1,0);
            \draw[color=wongblue,ultra thick] (0, -1) -- (0,0);
            \draw (0, 0) -- (0,1);
        \end{tikzpicture}
        &
        \begin{tikzpicture}[scale=0.5]
            \draw (-1, 0) -- (0,0);
            \draw[color=wongblue,ultra thick] (0, 0) -- (1,0);
            \draw (0, -1) -- (0,0);
            \draw[color=wongblue,ultra thick] (0, 0) -- (0,1);
        \end{tikzpicture}
        \\
        $a_1$
        &
        $a_1$
        &
        $b_1$
        &
        $b_1$
        &
        $c_1$
        &
        $c_1$
        \\
        \multicolumn{6}{c}{\textbf{H-D vertices}}\\
        \begin{tikzpicture}[scale=0.5]
            \node at (0, 1.5) {};
            \draw[color=wongblue,ultra thick] (-1, 0) -- (0,0);
            \draw[color=wongblue,ultra thick] (0, 0) -- (1,0);
            \draw[color=wongblue,ultra thick] (1, -1) -- (0,0);
            \draw[color=wongblue,ultra thick] (0, 0) -- (-1,1);
        \end{tikzpicture}
        &
        \begin{tikzpicture}[scale=0.5]
            \draw (-1,0) -- (0,0);
            \draw (0, 0) -- (1,0);
            \draw (1, -1) -- (0,0);
            \draw (0, 0) -- (-1,1);
        \end{tikzpicture}
        &
        \begin{tikzpicture}[scale=0.5]
            \draw[color=wongblue,ultra thick] (-1, 0) -- (0,0);
            \draw[color=wongblue,ultra thick] (0,0) -- (1,0);
            \draw (1, -1) -- (0,0);
            \draw (0, 0) -- (-1,1);
        \end{tikzpicture}
        &
        \begin{tikzpicture}[scale=0.5]
            \draw (-1, 0) -- (0,0);
            \draw (0, 0) -- (1,0);
            \draw[color=wongblue,ultra thick] (1, -1) -- (0,0);
            \draw[color=wongblue,ultra thick] (0, 0) -- (-1,1);
        \end{tikzpicture}
        &
        \begin{tikzpicture}[scale=0.5]
            \draw[color=wongblue,ultra thick] (-1, 0) -- (0,0);
            \draw (0, 0) -- (1,-0);
            \draw[color=wongblue,ultra thick] (1, -1) -- (0,0);
            \draw (0, 0) -- (-1,1);
        \end{tikzpicture}
        &
        \begin{tikzpicture}[scale=0.5]
            \draw (-1, 0) -- (0,0);
            \draw[color=wongblue,ultra thick] (0, 0) -- (1,0);
            \draw (1, -1) -- (0,0);
            \draw[color=wongblue,ultra thick] (0, 0) -- (-1,1);
        \end{tikzpicture}
        \\
        $a_2$
        &
        $a_2$
        &
        $b_2$
        &
        $b_2$
        &
        $c_2$
        &
        $c_2$
        \\
        \multicolumn{6}{c}{\textbf{V-D vertices}}\\
        \begin{tikzpicture}[scale=0.5]
            \node at (0, 1.5) {};
            \draw[color=wongblue,ultra thick] (-1, 1) -- (0,0);
            \draw[color=wongblue,ultra thick] (0, 0) -- (1,-1);
            \draw[color=wongblue,ultra thick] (0, -1) -- (0,0);
            \draw[color=wongblue,ultra thick] (0, 0) -- (0,1);
        \end{tikzpicture}
        &
        \begin{tikzpicture}[scale=0.5]
            \draw (-1, 1) -- (0,0);
            \draw (0, 0) -- (1,-1);
            \draw (0, -1) -- (0,0);
            \draw (0, 0) -- (0,1);
        \end{tikzpicture}
        &
        \begin{tikzpicture}[scale=0.5]
            \draw[color=wongblue,ultra thick] (-1, 1) -- (0,0);
            \draw[color=wongblue,ultra thick] (0, 0) -- (1,-1);
            \draw (0, -1) -- (0,0);
            \draw (0, 0) -- (0,1);
        \end{tikzpicture}
        &
        \begin{tikzpicture}[scale=0.5]
            \draw (-1, 1) -- (0,0);
            \draw (0, 0) -- (1,-1);
            \draw[color=wongblue,ultra thick] (0, -1) -- (0,0);
            \draw[color=wongblue,ultra thick] (0, 0) -- (0,1);
        \end{tikzpicture}
        &
        \begin{tikzpicture}[scale=0.5]
            \draw[color=wongblue,ultra thick] (-1, 1) -- (0,0);
            \draw (0, 0) -- (1,-1);
            \draw[color=wongblue,ultra thick] (0, -1) -- (0,0);
            \draw (0, 0) -- (0,1);
        \end{tikzpicture}
        &
        \begin{tikzpicture}[scale=0.5]
            \draw (-1, 1) -- (0,0);
            \draw[color=wongblue,ultra thick] (0, 0) -- (1,-1);
            \draw (0, -1) -- (0,0);
            \draw[color=wongblue,ultra thick] (0, 0) -- (0,1);
        \end{tikzpicture}
        \\
        $a_3$
        &
        $a_3$
        &
        $b_3$
        &
        $b_3$
        &
        $c_3$
        &
        $c_3$
    \end{tabular}
    \end{minipage}
    \hfill
    \begin{minipage}{0.35\textwidth}
    \begin{equation*}
        \begin{split}
            (a_1, b_1, c_1) &= (2^{-1/3}, 2^{1/6}, 2^{-1/3}), \\
            (a_2, b_2, c_2) &= q^3 (2^{1/6}, 2^{-1/3}, 2^{-1/3}), \\
            (a_3, b_3, c_3) &= q^{-3} (2^{1/6}, 2^{-1/3}, 2^{-1/3}), \\
            q &= e^{\pi i / 8}.
        \end{split} \label{eq:combinatorial-point-20V}
    \end{equation*}
    \end{minipage}

    \caption{The vertex weights of degree-$4$ vertices. Here, ``H-V'' means that the vertex is the intersection of a horizontal line and a vertical line. The abbreviations ``H-D'' and ``V-D'' are understood similarly, where ``D'' stands for ``diagonal.''} \label{tab:local-weights}
\end{table}

Next, we explain how Yang--Baxter moves are interpreted as probabilistic bijections.
Let $L$ and $M$ be two graphs obtained from $\mathcal{Q}_{\mathbf{k}}$ by Yang--Baxter moves, and suppose that $L$ and $M$ differ by a single Yang--Baxter move.
We identify the corresponding edges before and after the move, and require them to have the same orientation whenever they are oriented; see \cref{fig:L-and-M}.
Let $X$ and $Y$ denote the sets of admissible configurations on $L$ and $M$, respectively.

\begin{figure}[!htb]
    \centering
    \begin{tikzpicture}[scale=0.6]
        \draw (0, 4) -- (0,0) -- (4,-4) -- (4,4) -- cycle;
        \draw (2,2) circle (1);
        \draw (1,2) -- (3,2);
        \draw (2,1) -- (2,3);
        \draw[rounded corners, thick] (1.292,2.707) -- (1.5,2.5) -- (1.5,1.5) -- (2.5,1.5) -- (2.707,1.292);
        \node at (2,-4) {\Large$L$};

        \draw[<->] (5,2) -- (7, 2);

        \draw (8, 4) -- (8,0) -- (12,-4) -- (12,4) -- cycle;
        \draw (10,2) circle (1);
        \draw (9,2) -- (11,2);
        \draw (10,1) -- (10,3);
        \draw[rounded corners, thick] (9.292,2.707) -- (9.5,2.5) -- (10.5,2.5) -- (10.5,1.5) -- (10.707,1.292);
        \node at (10,-4) {\Large$M$};
    \end{tikzpicture}
    \caption{A schematic picture in which $L$ and $M$ are transformed into each other by a single Yang--Baxter move.}\label{fig:L-and-M}
\end{figure}

Let $x \in X$ and $y \in Y$.
We say that $y$ is reachable from $x$ if, when the Yang--Baxter move is applied to $x$, the configuration~$y$ has the same edge orientations as $x$ outside the domain where the Yang--Baxter move is applied.
Denote by $R(x) \subseteq Y$ all reachable configurations from $x$.
We analogously define the reachability from $y$ to $x$ and denote by $R(y)$ the set of reachable configurations from $y$.
See \cref{fig:R(a)-R(b)} for an example.
Note that $x \in R(y)$ if and only if $y \in R(x)$.
\begin{figure}[!htb]
    \centering
    \begin{tikzpicture}[scale=0.4]
        \begin{scope}
            \draw (-2, 0) -- (2, 0);
            \draw (0, -2) -- (0, 2);
            \draw[color=wongblue, very thick, rounded corners] (-2,2) -- (-1,1) -- (-1,-1) -- (1,-1) -- (2, -2);

            \draw[->] (3, 0) -- (4, 0);

            \draw (5, 0) -- (9, 0);
            \draw (7, -2) -- (7, 2);
            \draw[rounded corners] (5,2)--(6,1)--(8,1)--(8,-1)--(9,-2);
            \draw[color=wongblue, very thick,rounded corners] (5,2)--(6,1)--(7,1);
            \draw[color=wongblue, very thick] (7,1)--(7,0)--(8,0);
            \draw[color=wongblue, very thick,rounded corners] (8,0)--(8,-1)--(9,-2);

            \draw (10, 0) -- (14, 0);
            \draw (12, -2) -- (12, 2);
            \draw[color=wongblue, very thick,rounded corners] (10,2)--(11,1)--(12,1)--(13,1)--(13,-1)--(14,-2);

            \node[] at (-3.5,0) {\textbf{(a)}};
        \end{scope}
        \begin{scope}[xshift=21cm]
            \draw (-2, 0) -- (2, 0);
            \draw (0, -2) -- (0, 2);
            \draw[color=wongblue, very thick,rounded corners] (-2,2)--(-1,1)--(0,1)--(1,1)--(1,-1)--(2,-2);

            \draw[->] (3, 0) -- (4, 0);

            \draw (5, 0) -- (9, 0);
            \draw (7, -2) -- (7, 2);
            \draw[rounded corners] (5,2) -- (6,1) -- (6,-1) -- (8,-1) -- (9, -2);
            \draw[color=wongblue,very thick, rounded corners] (5,2)--(6,1)--(6,0);
            \draw[color=wongblue,very thick] (6,0)--(7,0)--(7,-1);
            \draw[color=wongblue,very thick, rounded corners] (7,-1)--(8,-1)--(9,-2);

            \draw (10, 0) -- (14, 0);
            \draw (12, -2) -- (12, 2);
            \draw[color=wongblue, very thick, rounded corners] (10,2) -- (11,1) -- (11,-1) -- (13,-1) -- (14, -2);

            \node[] at (-3.5,0) {\textbf{(b)}};
        \end{scope}
    \end{tikzpicture}
    \caption{In (a), a Yang--Baxter move locally flips the diagonal line from the southwest side to the northeast side. Under this move, the configuration on the left has two reachable configurations, shown on the right. In (b), a Yang--Baxter move locally flips the diagonal line from the northeast side to the southwest side. Under this move, the configuration on the left has two reachable configurations, shown on the right.}\label{fig:R(a)-R(b)}
\end{figure}

The following proposition shows that each Yang--Baxter move induces a probabilistic bijection.
\begin{prop} \label{prop:Yang--Baxter-moves-as-probabilistic-bijections}
    Define the maps~$\mathcal{P}, \overline{\mathcal{P}} \colon X \times Y \to \CC$ as follows:
    \begin{align*}
        \mathcal{P}(x, y) &= \begin{cases}
        \omega_M(y) \cdot \rbra*{\sum\limits_{y' \in R(x)} \omega_M(y')}^{-1} & \text{if}\ y \in R(x), \\
            0 & \text{if}\ y \notin R(x),
        \end{cases}\\
        \overline{\mathcal{P}}(x, y) &= \begin{cases}
        \omega_L(x) \cdot \rbra*{\sum\limits_{x' \in R(y)} \omega_L(x')}^{-1} & \text{if}\ x \in R(y),\\
            0 & \text{if}\ x \notin R(y).
        \end{cases}
    \end{align*}
    Then the pair~$(\mathcal{P}, \overline{\mathcal{P}})$ is a probabilistic bijection from $(X, \omega_L)$ to $(Y, \omega_M)$.
\end{prop}
\begin{proof}
    It is straightforward to check the first two conditions of \cref{defi:probabilistic-bijections}.
    For the third condition, let $x \in X$ and $y \in Y$.
    If $y \notin R(x)$, then $x \notin R(y)$ and both sides of the equation are zero.
    If $y \in R(x)$, then $x \in R(y)$ and we have
    \begin{equation*}
        \omega_L(x) \mathcal{P}(x, y) = \omega_L(x) \cdot \omega_M(y) \cdot \rbra*{\sum_{y' \in R(x)} \omega_M(y')}^{-1} = \omega_M(y) \cdot \omega_L(x) \cdot \rbra*{\sum_{x' \in R(y)} \omega_L(x')}^{-1}
        = \omega_M(y) \overline{\mathcal{P}}(x, y),
    \end{equation*}
    where the second equality follows from the Yang--Baxter equation for the local domain where the Yang--Baxter move is applied.
\end{proof}

\subsection{A probabilistic bijection from \texorpdfstring{$20$}{20}V configurations to mixed \texorpdfstring{$6$}{6}V configurations}

It is shown in \cite[Section 2.3]{MR4395233} that the domain~$\mathcal{Q}_{(1, 2, \ldots, n)}$ can be transformed to the domain~$\mathcal{M}_{(1, 2, \ldots, n)}$ with some frozen domains attached to it by applying a sequence of Yang--Baxter moves.

We generalize this transformation to the case of the domains~$\mathcal{Q}_\mathbf{k}$ and $\mathcal{M}_\mathbf{k}$ for arbitrary strictly increasing sequences~$\mathbf{k} = (k_1, k_2, \ldots, k_n)$.
We illustrate the transformation from $\mathcal{Q}_\mathbf{k}$ to $\mathcal{M}_\mathbf{k}$ in \cref{fig:20V-to-mixed-6V}, where the horizontal, vertical, and diagonal lines in the initial domain $\mathcal{Q}_\mathbf{k}$ are labeled by $z_i$, $w_i$, and $t_i$, respectively.
Namely, the transformation from $\mathcal{Q}_\mathbf{k}$ to $\mathcal{M}_\mathbf{k}$ is achieved by repeatedly applying Yang--Baxter moves so as to lift and straighten the diagonal lines into ``hooks''.
During this process, several frozen domains (i.e., domains with a unique admissible configuration) appear naturally; these are denoted by $F_1,\dots,F_4$ in Figure~\ref{fig:20V-to-mixed-6V}.
In the resulting mixed 6V domain $\mathcal{M}_\mathbf{k}$, vertices in odd rows are intersections of $z_i$- and $w_i$-lines, while vertices in even rows are intersections of $t_i$- and $w_i$-lines.

Figure~\ref{fig:20V-to-mixed-6V-local} illustrates how a diagonal line is locally transformed.
Each Yang--Baxter move lifts the diagonal line locally.
Repeated application of these moves converts the diagonal into a ``hook''.

\begin{figure}[htb]
    \centering
    \includegraphics[scale=0.59]{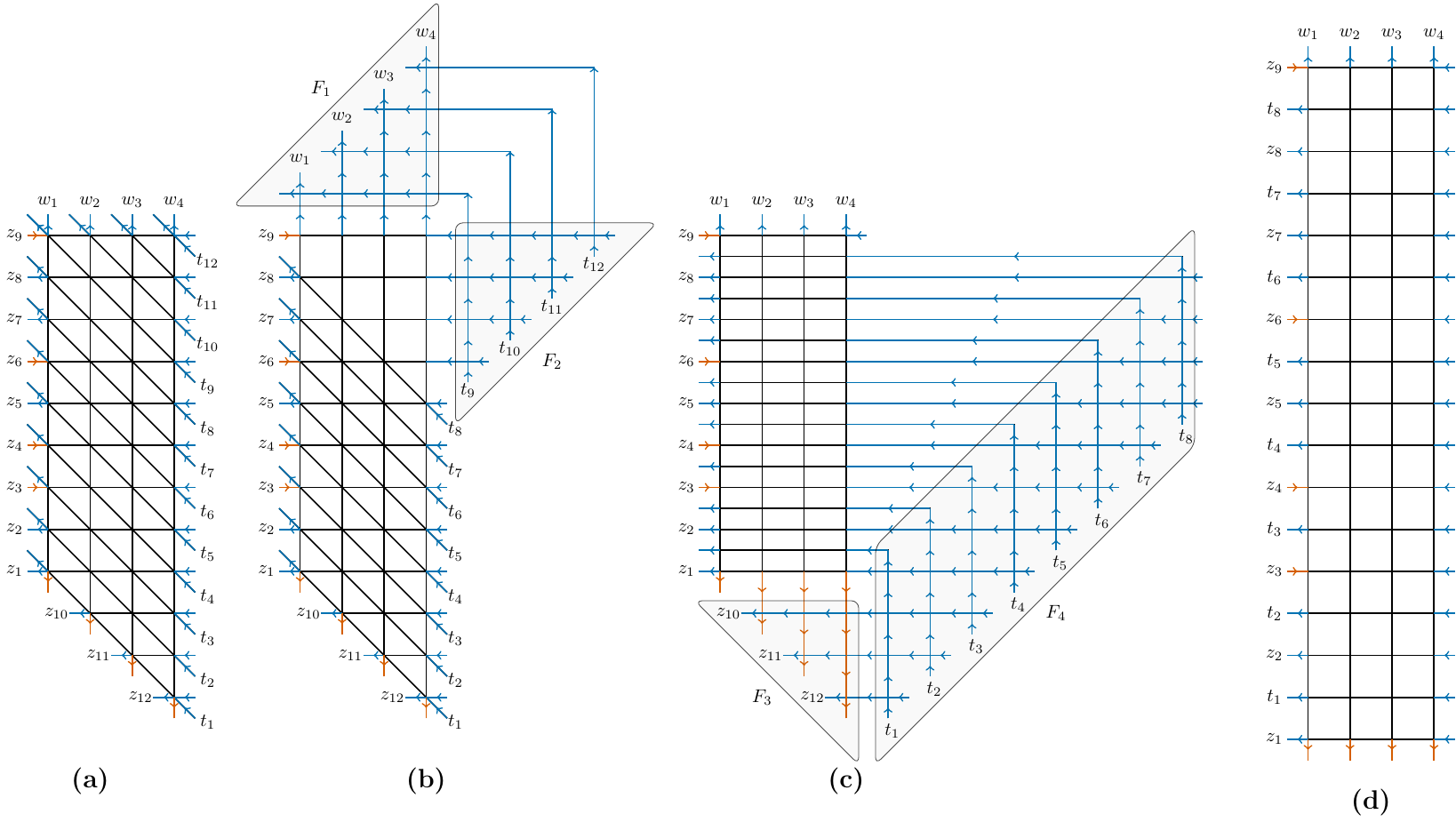}
    \caption{Transformation from the 20V domain $\mathcal \mathcal{Q}_\mathbf{k}$ to the mixed 6V domain $\mathcal \mathcal{M}_\mathbf{k}$ for $\mathbf{k}=(3,4,6,9)$. (a) The initial domain $\mathcal{Q}_\mathbf{k}$. (b) The domain obtained by lifting the top $n$ diagonal lines upward via Yang--Baxter moves. (c) The domain obtained from (b) after removing the frozen domains $F_1$ and $F_2$ and lifting the remaining diagonal lines upward. (d) The mixed 6V domain $\mathcal{M}_\mathbf{k}$, obtained after removing the frozen domains $F_3$ and $F_4$ in (c).} \label{fig:20V-to-mixed-6V}
\end{figure}
\begin{figure}[htb]
    \centering
    \includegraphics[scale=0.7]{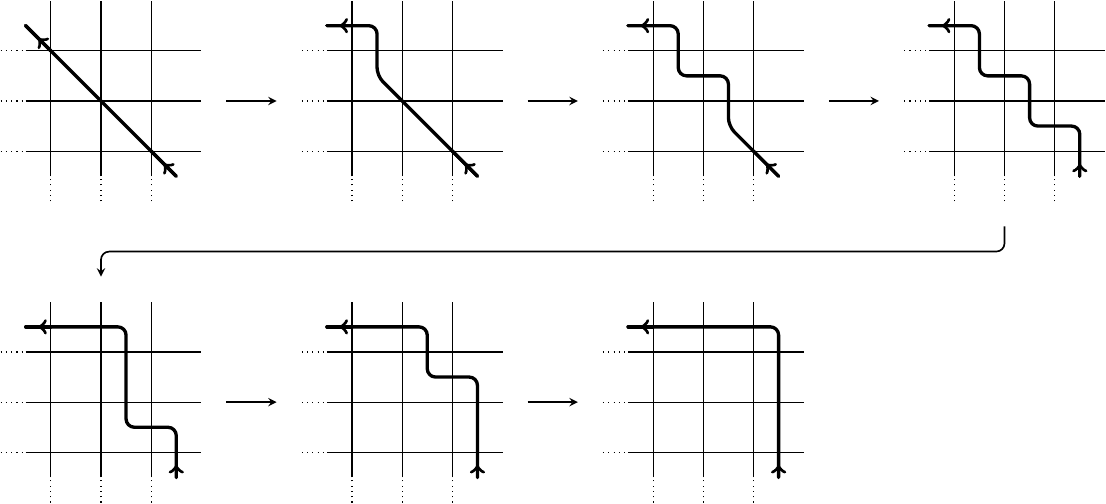}
    \caption{Local lifting and straightening procedure for the diagonal line~$t_{10}$. Starting from the diagonal line, repeated Yang--Baxter moves lift the line upward and eventually straighten it into a ``hook''. For simplicity, the other diagonal lines are omitted.} \label{fig:20V-to-mixed-6V-local}
\end{figure}
When these frozen domains are removed, their weights factor out as
constants independent of the remaining mixed 6V configuration.
For example, the frozen domain $F_1$ consists of $n(n+1)/2$ vertices
of weight $a_3$, and hence contributes the factor
$a_3^{n(n+1)/2}$.
Multiplying the contributions from $F_1,\ldots,F_4$, we obtain the constant factor
\begin{equation}
    \mathcal{C}_\mathbf{k} \coloneqq b_1^{n(n-1)/2} a_2^{n(n-1)/2+n k_n} a_3^{n(n+1)/2}.\label{eq:frozen-domains-constant}
\end{equation}

Here, we clarify how the weight of a mixed $6$V configuration is calculated.
For a mixed $6$V configuration, its weight is the product of the weights of the vertices, with the vertex weights given in \cref{tab:local-weights}.
More precisely, a vertex in an odd row contributes one of $a_1,b_1,c_1$, while a vertex in an even row contributes one of $a_3,b_3,c_3$.
This convention arises from how the lines in the domain~$\mathcal{Q}_{\mathbf{k}}$ are transformed: in the even-row case, we match the local configuration with one in the third row of \cref{tab:local-weights} by rotating the diagonal line by $45^\circ$ counterclockwise.
See \cref{exam:inv-crossing} for an example.
All possible vertex types appearing in mixed $6$V configurations, together with their weights, are listed in \cref{tab:m6V-local-weights}.
Since this weighting is consistent with the weight~$\omega_G(x)$ defined above, we also write $\omega_{\mathcal{M}_{\mathbf{k}}}(x)$ for the weight of a mixed $6$V configuration~$x$.

By \cref{lemma:probabilistic-bijection-composition,prop:Yang--Baxter-moves-as-probabilistic-bijections}, a sequence of Yang--Baxter moves induces a probabilistic bijection between the sets of admissible configurations on the initial and final domains.
Applying this to the Yang--Baxter moves transforming the domain~$\mathcal{Q}_\mathbf{k}$ to the domain~$\mathcal{M}_\mathbf{k}$, we obtain the following proposition.
\begin{prop}\label{prop:twenty-mixed6V-prob-bij}
    Let $\mathbf{k} = (k_1, k_2, \ldots, k_n)$ be a strictly increasing sequence of positive integers.
    We choose the vertex weights as in \cref{tab:local-weights}.
    Then, there exists a probabilistic bijection between the set of $20$V configurations on the domain~$\mathcal{Q}_\mathbf{k}$ with the constant weight function~$1$ and the set of mixed $6$V configurations on the domain~$\mathcal{M}_\mathbf{k}$ with the weight function~$\mathcal{C}_\mathbf{k} \cdot \omega_{\mathcal{M}_\mathbf{k}}(\cdot)$.
\end{prop}

\begin{table}[htb]
    \centering
    \begin{tabular}{c||c|c|c|c|c|c}
        Vertex type
        &
        \begin{tikzpicture}[scale=0.5]
            \draw[color=wongblue,ultra thick] (-1, 0) -- (0,0);
            \draw[color=wongblue,ultra thick] (0, 0) -- (1,0);
            \draw[color=wongblue,ultra thick] (0, -1) -- (0,0);
            \draw[color=wongblue,ultra thick] (0, 0) -- (0,1);
        \end{tikzpicture}
        &
        \begin{tikzpicture}[scale=0.5]
            \draw (-1, 0) -- (0,0);
            \draw (0, 0) -- (1,0);
            \draw (0, -1) -- (0,0);
            \draw (0, 0) -- (0,1);
        \end{tikzpicture}
        &
        \begin{tikzpicture}[scale=0.5]
            \draw[color=wongblue,ultra thick] (-1, 0) -- (0,0);
            \draw[color=wongblue,ultra thick] (0, 0) -- (1,0);
            \draw (0, -1) -- (0,0);
            \draw (0, 0) -- (0,1);
        \end{tikzpicture}
        &
        \begin{tikzpicture}[scale=0.5]
            \draw (-1, 0) -- (0,0);
            \draw (0, 0) -- (1,0);
            \draw[color=wongblue,ultra thick] (0, -1) -- (0,0);
            \draw[color=wongblue,ultra thick] (0, 0) -- (0,1);
        \end{tikzpicture}
        &
        \begin{tikzpicture}[scale=0.5]
            \draw[color=wongblue,ultra thick] (-1, 0) -- (0,0);
            \draw (0, 0) -- (1,0);
            \draw[color=wongblue,ultra thick] (0, -1) -- (0,0);
            \draw (0, 0) -- (0,1);
        \end{tikzpicture}
        &
        \begin{tikzpicture}[scale=0.5]
            \draw (-1, 0) -- (0,0);
            \draw[color=wongblue,ultra thick] (0, 0) -- (1,0);
            \draw (0, -1) -- (0,0);
            \draw[color=wongblue,ultra thick] (0, 0) -- (0,1);
        \end{tikzpicture}
        \\
        \hline
        Odd row
        &
        $a_1$
        &
        $a_1$
        &
        $b_1$
        &
        $b_1$
        &
        $c_1$
        &
        $c_1$
        \\
        \hline
        Even row
        &
        $a_3$
        &
        $a_3$
        &
        $b_3$
        &
        $b_3$
        &
        $c_3$
        &
        $c_3$
    \end{tabular}
    \caption{Vertex weights for mixed $6$V configurations. Here, the weights depend on the parity of the row containing the vertex. The rows are numbered $1,2,\ldots,2k_n-1$ from bottom to top.} \label{tab:m6V-local-weights}
\end{table}

\section{The weights of mixed \texorpdfstring{$6$}{6}V configurations in terms of the variant inversion number} \label{sec:comb-desc}

In this section, we express the weights of mixed 6V configurations in terms of the variant inversion number, an integer-valued statistic.
In \cref{sec:map-mixed-6V-to-amtri}, this expression will be used to prove a relation between the weights of mixed $6$V configurations and those of triple-free GT patterns (\cref{theo:psi-is-correct-map}).

Let $x$ be a mixed $6$V configuration on the domain~$\mathcal{M}_\mathbf{k}$.
We assign a type to each local configuration at a vertex according to the dictionary in \cref{fig:vertex-types}.
Let $\mathcal{N}_{(j)}^{i}(x)$ denote the number of type-$j$ configurations in the $i$-th row of $x$.
Define $\mathcal{N}_{(j)}^{\odd}(x)$ and $\mathcal{N}_{(j)}^{\even}(x)$ by
\begin{equation*}
    \mathcal{N}_{(j)}^{\odd}(x) = \sum_{i=1}^{k_n} \mathcal{N}_{(j)}^{2i-1}(x) \quad\text{and}\quad \mathcal{N}_{(j)}^{\even}(x) = \sum_{i=1}^{k_n-1} \mathcal{N}_{(j)}^{2i}(x).
\end{equation*}

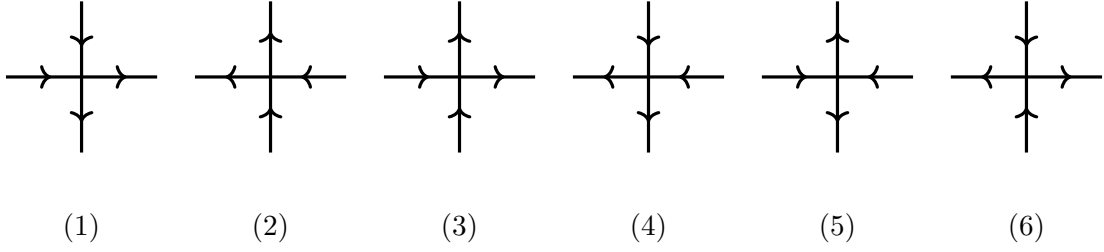
\begin{figure}[htb]
    \centering
    \begin{tikzpicture}
        \begin{scope}[very thick,decoration={
        markings,
        mark=at position 0.6 with {\arrow{>}}}
        ]
            \draw[postaction={decorate}] (-1,0) -- (0,0);
            \draw[postaction={decorate}] (0,0) -- (1,0);
            \draw[postaction={decorate}] (0,1) -- (0,0);
            \draw[postaction={decorate}] (0,0) -- (0,-1);
            \node at (0,-2) {\large(1)};
        \end{scope}
        \begin{scope}[shift={(2.5,0)},
        very thick,decoration={
        markings,
        mark=at position 0.6 with {\arrow{>}}}
        ]
            \draw[postaction={decorate}] (0,0) -- (-1,0);
            \draw[postaction={decorate}] (1,0) -- (0,0);
            \draw[postaction={decorate}] (0,0) -- (0,1);
            \draw[postaction={decorate}] (0,-1) -- (0,0);
            \node at (0,-2) {\large(2)};
        \end{scope}
        \begin{scope}[shift={(5,0)},
        very thick,decoration={
        markings,
        mark=at position 0.6 with {\arrow{>}}}
        ]
            \draw[postaction={decorate}] (-1,0) -- (0,0);
            \draw[postaction={decorate}] (0,0) -- (1,0);
            \draw[postaction={decorate}] (0,0) -- (0,1);
            \draw[postaction={decorate}] (0,-1) -- (0,0);
            \node at (0,-2) {\large(3)};
        \end{scope}
        \begin{scope}[shift={(7.5,0)},
        very thick,decoration={
        markings,
        mark=at position 0.6 with {\arrow{>}}}
        ]
            \draw[postaction={decorate}] (0,0) -- (-1,0);
            \draw[postaction={decorate}] (1,0) -- (0,0);
            \draw[postaction={decorate}] (0,1) -- (0,0);
            \draw[postaction={decorate}] (0,0) -- (0,-1);
            \node at (0,-2) {\large(4)};
        \end{scope}
        \begin{scope}[shift={(10,0)},
        very thick,decoration={
        markings,
        mark=at position 0.6 with {\arrow{>}}}
        ]
            \draw[postaction={decorate}] (-1,0) -- (0,0);
            \draw[postaction={decorate}] (1,0) -- (0,0);
            \draw[postaction={decorate}] (0,0) -- (0,1);
            \draw[postaction={decorate}] (0,0) -- (0,-1);
            \node at (0,-2) {\large(5)};
        \end{scope}
        \begin{scope}[shift={(12.5,0)},
        very thick,decoration={
        markings,
        mark=at position 0.6 with {\arrow{>}}}
        ]
            \draw[postaction={decorate}] (0,0) -- (-1,0);
            \draw[postaction={decorate}] (0,0) -- (1,0);
            \draw[postaction={decorate}] (0,1) -- (0,0);
            \draw[postaction={decorate}] (0,-1) -- (0,0);
            \node at (0,-2) {\large(6)};
        \end{scope}
    \end{tikzpicture}
    \caption{The $6$ possible vertex configurations and their types.} \label{fig:vertex-types}
\end{figure}

\begin{defi}
    We define the \myemph{variant inversion number}~$\ic(x)$ of a mixed $6$V configuration~$x$ as follows:
    \begin{equation*}
        \ic(x) = \mathcal{N}_{(1)}^{\even}(x) + \mathcal{N}_{(3)}^{\odd}(x).
    \end{equation*}
\end{defi}

\begin{rema} \label{rema:why-variant-inversion-number}
    The name ``variant inversion number'' comes from the fact that $\ic(x) = \inv(A) + \mathcal{N}_{(1)}^{\even}(A) - \mathcal{N}_{(3)}^{\even}(A)$ holds.
    Here, $A$ is the $(2k_n-1)\times n$-matrix with entries in $\{0,\pm1\}$ obtained from $x$ by interpreting the type-$5$ vertices as $1$, the type-$6$ vertices as $-1$, and the other vertices as $0$, and $\inv(A)$ is the generalized inversion number defined by $\inv(A) = \sum_{1 \leq i' < i \leq 2k_n-1, 1 \leq j' \leq j \leq n} A_{i', j} A_{i, j'}$.
    This fact can be proved by an argument similar to the proof of \cref{theo:inv-crossing} below by calculating the changes of the inversion number~$\inv(x)$ in addition to the other quantities.

    Experiments suggest that the distribution of the statistic~$\ic(x)$ may be identical to that of $\inv(x)$.
    In other words, for each non-negative integer $i$, the number of mixed $6$V configurations $x$ on the domain~$\mathcal{M}_\mathbf{k}$ with $\ic(x) = i$ equals the number of those with $\inv(x) = i$.
    We verified this observation by computer for $\mathbf{k} = (1,2,\ldots,n)$ with $1 \leq n \leq 5$.
\end{rema}

The following theorem gives an expression for the weights of mixed $6$V configurations in terms of the variant inversion number.

\begin{theo} \label{theo:inv-crossing}
    Let $\mathbf{k} = (k_1, k_2, \ldots, k_n)$ be a strictly increasing sequence of positive integers.
    Let $x$ be a mixed $6$V configuration on the domain~$\mathcal{M}_\mathbf{k}$.
    Then, we have
    \begin{equation}
        \mathcal{C}_\mathbf{k} \cdot \omega_{\mathcal{M}_\mathbf{k}}(x) = 2^{\ic(x)}. \label{eq:comb-desc}
    \end{equation}
    (See \cref{eq:frozen-domains-constant} for the definition of $\mathcal{C}_\mathbf{k}$.)
\end{theo}

\begin{exam} \label{exam:inv-crossing}
    Let us illustrate \cref{theo:inv-crossing} with an example.
    Let $x$ be the following configuration on the domain~$\mathcal{M}_{(1,2,3,4)}$: \begin{center}
        \begin{tikzpicture}[scale=0.6,baseline=(current bounding box.center)]
            \draw (1,1) -- (4,1);
            \draw (1,2) -- (4,2);
            \draw (1,3) -- (4,3);
            \draw (1,4) -- (4,4);
            \draw (1,5) -- (4,5);
            \draw (1,6) -- (4,6);
            \draw (1,7) -- (4,7);

            \draw (1,1) -- (1,7);
            \draw (2,1) -- (2,7);
            \draw (3,1) -- (3,7);
            \draw (4,1) -- (4,7);

            \draw[color=wongblue, ultra thick] (0.5,1) -- (1,1)--(1,0.5);
            \draw[color=wongblue, ultra thick] (0.5,3) -- (1,3)--(2,3)--(2,1)--(2,0.5);
            \draw[color=wongblue, ultra thick] (0.5,5)--(1,5)--(1,4)--(2,4)--(2,3)--(3,3)--(3,1)--(3,0.5);
            \draw[color=wongblue, ultra thick] (0.5,7)--(1,7)--(2,7)--(2,4)--(4,4)--(4,1)--(4,0.5);
        \end{tikzpicture}.
    \end{center}
    We have $\mathcal{N}_{(1)}^{\even}(x) = 1$ and $\mathcal{N}_{(3)}^{\odd}(x) = 2$. Thus, we compute $\ic(x) = 1 + 2 = 3$.
    Meanwhile, the prefactor~$\mathcal{C}_\mathbf{k}$ for $\mathbf{k} = (1, 2, 3, 4)$ is
    \begin{equation*}
        \mathcal{C}_\mathbf{k}
        = b_1^{4\cdot 3 / 2} a_2^{4\cdot 11 / 2} a_3^{4\cdot 5 / 2}
        = 2^{6+1/3}q^4,
    \end{equation*}
    and the weight~$\omega_{\mathcal{M}_{(1,2,3,4)}}(x)$ of the configuration $x$ is
    \begin{align*}
        \omega_{\mathcal{M}_{(1,2,3,4)}}(x) &= b_1 c_1 a_1 a_1 \cdot a_3 b_3 a_3 a_3 \cdot c_1 b_1 a_1 a_1 \cdot c_3 a_3 b_3 c_3 \cdot b_1 a_1 c_1 b_1 \cdot a_3 b_3 b_3 b_3 \cdot c_1 b_1 b_1 b_1 \\
        &= a_1^5 b_1^7 c_1^4 \times a_3^5 b_3^5 c_3^2 = 2^{-3-1/3} q^{-36}.
    \end{align*}
    Thus, $\mathcal{C}_\mathbf{k} \omega_{\mathcal{M}_{(1,2,3,4)}}(x) = 2^{6+1/3}q^4 \times 2^{-3-1/3} q^{-36} = 2^{3}$ by the periodicity~$q^{16} = 1$.
    Hence, we have
    \begin{equation*}
        \mathcal{C}_\mathbf{k} \omega_{\mathcal{M}_{(1,2,3,4)}}(x) = 2^3 = 2^{\ic(x)},
    \end{equation*}
    which is consistent with \cref{theo:inv-crossing}.
\end{exam}

\begin{figure}[bht]
    \centering
    \begin{minipage}{.48\textwidth}
        \centering
        \begin{tikzpicture}[scale=0.7]
            \draw (1,1) -- (4,1);
            \draw (1,2) -- (4,2);
            \draw (1,3) -- (4,3);
            \draw (1,4) -- (4,4);
            \draw (1,5) -- (4,5);
            \draw (1,6) -- (4,6);
            \draw (1,7) -- (4,7);
            \draw (1,8) -- (4,8);
            \draw (1,9) -- (4,9);
            \draw (1,10) -- (4,10);
            \draw (1,11) -- (4,11);

            \draw (1,1) -- (1,11);
            \draw (2,1) -- (2,11);
            \draw (3,1) -- (3,11);
            \draw (4,1) -- (4,11);

            \draw[color=wongblue, ultra thick] (0.5,3)--(1,3)--(1,0.5);
            \draw[color=wongblue, ultra thick] (0.5,5)--(2,5)--(2,0.5);
            \draw[color=wongblue, ultra thick] (0.5,7)--(3,7)--(3,0.5);
            \draw[color=wongblue, ultra thick] (0.5,11)--(4,11)--(4,0.5);
        \end{tikzpicture}
        \caption{The base configuration $x_{\max}$ for $\mathbf{k}=(2,3,4,6)$.} \label{fig:b-max}
    \end{minipage}
    \begin{minipage}{.48\textwidth}
        \centering
        \begin{tikzpicture}[scale=0.9]
            \draw (1,2) -- (4,2);
            \draw (1,3) -- (4,3);
            \draw (2,1) -- (2,4);
            \draw (3,1) -- (3,4);
            \draw [color=wongblue, ultra thick] (2,3) -- (3,3) -- (3,2);

            \draw[->] (4.5, 2.5) -- (5.5, 2.5);

            \draw (6,2) -- (9,2);
            \draw (6,3) -- (9,3);
            \draw (7,1) -- (7,4);
            \draw (8,1) -- (8,4);
            \draw [color=wongblue, ultra thick] (7,3) -- (7,2) -- (8,2);
        \end{tikzpicture}
        \caption{Flipping a corner.} \label{fig:turning-over-corner}
    \end{minipage}
\end{figure}

\begin{proof}[Proof of \cref{theo:inv-crossing}]
    Let $x_{\max}$ be the configuration in which the $i$-th path starts at $(1, 2k_i-1)$, moves straight to the right until reaching $(i, 2k_i-1)$, then goes straight downward, and ends at $(i, 1)$, for $i = 1, 2, \ldots, n$.
    We show an example of $x_{\max}$ in \cref{fig:b-max}.
    It is easy to see that any configuration on the domain~$\mathcal{M}_\mathbf{k}$ can be obtained from $x_{\max}$ by repeatedly flipping corners in the manner described in \cref{fig:turning-over-corner}.
    We show the theorem by induction on the number of times corners are flipped, starting from $x_{\max}$.

    First, assume $x = x_{\max}$.
    Then, it is straightforward to check $2^{\ic(x_{\max})} = \mathcal{C}_\mathbf{k} \times \omega_{\mathcal{M}_\mathbf{k}}(x_{\max}) = 2^{n(n-1)/2}$.
    Hence, \cref{eq:comb-desc} holds for $x_{\max}$.

    Next, suppose that \cref{eq:comb-desc} holds for a configuration~$x$ and that $x$ is transformed to $x'$ by flipping the corner around a square~$C$ in $x$.
    Then it suffices to show $2^{\ic(x')-\ic(x)} = \omega_{\mathcal{M}_\mathbf{k}}(x')/\omega_{\mathcal{M}_\mathbf{k}}(x)$.
    Assume that $C$ consists of four vertices~$(i, j), (i, j+1), (i+1, j), (i+1, j+1)$ as in the following figure. \begin{center}
        \begin{tikzpicture}[scale=0.8]
            \node at (0,2.5) {$C: $};
            \draw (1,2) -- (4,2);
            \draw (1,3) -- (4,3);
            \draw (2,1) -- (2,4);
            \draw (3,1) -- (3,4);
            \node at (4.5,2) {$j$};
            \node at (4.7,3) {$j+1$};
            \node at (2,0.5) {$i$};
            \node at (3,0.5) {$i+1$};
        \end{tikzpicture}
    \end{center}
    Furthermore, let
    \begin{align*}
        s &= -[\text{the vertex at $(i,j+1)$ is type-$6$}],
        \quad t = [\text{the vertex at $(i+1,j+1)$ is type-$5$}],\\
        u &= [\text{the vertex at $(i,j)$ is type-$5$}],
        \quad v = -[\text{the vertex at $(i+1,j)$ is type-$6$}].
    \end{align*}
    (In other words, $s, t, u, v$ are the matrix entries~$A_{2k_n-j-1, i},A_{2k_n-j-1,i+1},A_{2k_n-j,i},A_{2k_n-j,i+1}$ of the $\{0,\pm1\}$-matrix~$A$ corresponding to $x$ obtained as in \cref{rema:why-variant-inversion-number}.)
    Then, the following follows from a simple case-by-case analysis based on the possible local configurations on the corner~$C$:
    \begin{align}
        \mathcal{N}_{(1)}^{\even}(x')-\mathcal{N}_{(1)}^{\even}(x) &= \begin{cases}
            [u=1] = u & \text{if }i\ \text{is even}; \\
            -[t=0] = t-1 & \text{if } i\ \text{is odd};
        \end{cases} \label{eq:dc-diff} \\
        \mathcal{N}_{(3)}^{\odd}(x') - \mathcal{N}_{(3)}^{\odd}(x) &= \begin{cases}
            -[s=0]=-s-1 & \text{if } i\ \text{is even}; \\
            -[v=-1]=-v & \text{if } i\ \text{is odd},
        \end{cases} \label{eq:uc-diff}
    \end{align}
    where $[\cdot]$ denotes the Iverson bracket.
    Hence by \cref{eq:dc-diff,eq:uc-diff}, we obtain
    \begin{equation*}
        \ic(x')-\ic(x) = \begin{cases}
            -s+u-1 & \text{if } i\ \text{is even}; \\
            t-v-1 & \text{if } i\ \text{is odd}.
        \end{cases}
    \end{equation*}

    Next, we check $\omega_{\mathcal{M}_\mathbf{k}}(x')/\omega_{\mathcal{M}_\mathbf{k}}(x)$. For that, we look at the changes in the weights for each vertex of $(i, j), (i, j+1), (i+1, j), (i+1, j+1)$.

    For illustration, suppose that $j$ is even.
    If the local configuration around the vertex~$(i, j+1)$ in $x$ is \tikz[scale=0.3]{\draw (2,1)--(2,3); \draw (1,2)--(3,2); \draw[color=wongblue, very thick] (2,3)--(2,2)--(3,2);}, the local configuration around the same vertex in $x'$ is \tikz[scale=0.3]{\draw (2,1)--(2,3); \draw (1,2)--(3,2); \draw[color=wongblue, very thick] (2,3)--(2,1);}.
    Thus the weight at the vertex~$(i, j+1)$ changes from $c_1 = 2^{-1/3}$ to $b_1 = 2^{1/6}$.
    If the local configuration around the vertex~$(i, j+1)$ in $x$ is \tikz[scale=0.3]{\draw (2,1)--(2,3); \draw (1,2)--(3,2); \draw[color=wongblue, very thick] (1,2)--(3,2);}, the local configuration around the same vertex in $x'$ is \tikz[scale=0.3]{\draw (2,1)--(2,3); \draw (1,2)--(3,2); \draw[color=wongblue, very thick] (1,2)--(2,2)--(2,1);}.
    Thus the weight at the vertex~$(i, j+1)$ changes from $b_1 = 2^{1/6}$ to $c_1 = 2^{-1/3}$.
    In both cases, if the weight at the vertex~$(i, j+1)$ in $x$ is $2^\alpha$, it changes to $2^{-1/6-\alpha}$ after flipping the corner.

    In a similar manner, we can track the changes in the weights at all four vertices, both when $j$ is even and when $j$ is odd.
    We summarize all cases in \cref{tab:changes-of-weights}.
    Denote by $\alpha, \beta, \gamma, \delta$ the exponents of $2$ in the weights at the vertices~$(i, j+1), (i+1, j+1), (i, j), (i+1, j)$, respectively.
    From \cref{tab:changes-of-weights}, it is easy to check that by flipping the corner, these exponents change in the following way:
    \begin{itemize}
        \item at $(i, j+1)$, $\alpha \mapsto -1/6-\alpha$ if $j$ is even, and $\alpha \mapsto \alpha$ if $j$ is odd;
        \item at $(i+1, j+1)$, $\beta \mapsto \beta$ if $j$ is even, and $\beta \mapsto -1/6-\beta$ if $j$ is odd;
        \item at $(i, j)$, $\gamma \mapsto -1/6-\gamma$ if $j$ is even, and $\gamma \mapsto \gamma$ if $j$ is odd;
        \item at $(i+1, j)$, $\delta \mapsto \delta$ if $j$ is even, and $\delta \mapsto -1/6-\delta$ if $j$ is odd.
    \end{itemize}

    Finally, we calculate the ratio~$\omega_{\mathcal{M}_\mathbf{k}}(x')/\omega_{\mathcal{M}_\mathbf{k}}(x)$.
    We consider two cases depending on the parity of $j$.
    First, assume $j$ is even.
    Then, we have
    \begin{equation}
        \frac{\omega_{\mathcal{M}_\mathbf{k}}(x')}{\omega_{\mathcal{M}_\mathbf{k}}(x)} = \frac{2^{-1/3-\alpha-\gamma}}{2^{\alpha+\gamma}} = 2^{-2\alpha-2\gamma-1/3}. \label{eq:omega-even}
    \end{equation}
    Notice that $\alpha$ and $s$ are in the relation~$\alpha=1/6+s/2$ and that $\gamma$ and $u$ are in the relation~$\gamma=1/6-u/2$.
    These relations can be shown by a simple case-by-case based on the possible configurations at the vertices~$(i, j+1), (i, j)$.
    By substituting these relations into \cref{eq:omega-even}, we obtain
    \begin{equation*}
        \frac{\omega_{\mathcal{M}_\mathbf{k}}(x')}{\omega_{\mathcal{M}_\mathbf{k}}(x)} = 2^{-2(1/6+s/2)-2(1/6-u/2)-1/3} = 2^{-s+u-1} = 2^{\ic(x')-\ic(x)}.
    \end{equation*}
    Analogously, if $j$ is odd, we obtain
    \begin{equation*}
        \frac{\omega_{\mathcal{M}_\mathbf{k}}(x')}{\omega_{\mathcal{M}_\mathbf{k}}(x)} = 2^{-2(1/6-t/2)-2(1/6+v/2)-1/3} = 2^{t-v-1} = 2^{\ic(x')-\ic(x)}.
    \end{equation*}
    This completes the proof.
\end{proof}

\begin{table}[htb]
    \begin{tabular}{|c|c|c|c|}
        \hline
        & local configuration & changes if $j$ is even & changes if $j$ is odd \\
        \hline
        \multirow{2}{*}{$(i, j+1)$} & \tikz[scale=0.3]{\draw (2,1)--(2,3); \draw (1,2)--(3,2); \draw[color=wongblue, very thick] (2,3)--(2,2)--(3,2);} $\mapsto$ \tikz[scale=0.3]{\draw (2,1)--(2,3); \draw (1,2)--(3,2); \draw[color=wongblue, very thick] (2,3)--(2,1);} & $c_1=2^{-1/3} \mapsto b_1 = 2^{1/6}$ & $c_3=q^{-3}2^{-1/3} \mapsto b_3=q^{-3}2^{-1/3}$ \\
        & \tikz[scale=0.3]{\draw (2,1)--(2,3); \draw (1,2)--(3,2); \draw[color=wongblue, very thick] (1,2)--(3,2);} $\mapsto$ \tikz[scale=0.3]{\draw (2,1)--(2,3); \draw (1,2)--(3,2); \draw[color=wongblue, very thick] (1,2)--(2,2)--(2,1);} & $b_1 = 2^{1/6} \mapsto c_1=2^{-1/3}$ & $b_3=q^{-3}2^{-1/3} \mapsto c_3=q^{-3}2^{-1/3}$ \\
        \hline
        \multirow{2}{*}{$(i+1, j+1)$} & \tikz[scale=0.3]{\draw (2,1)--(2,3); \draw (1,2)--(3,2); \draw[color=wongblue, very thick] (1,2)--(2,2)--(2,1); \draw[color=wongblue, very thick] (2,3)--(2,2)--(3,2);} $\mapsto$ \tikz[scale=0.3]{\draw (2,1)--(2,3); \draw (1,2)--(3,2); \draw[color=wongblue, very thick] (2,3)--(2,2)--(3,2);} & $a_1=2^{-1/3} \mapsto c_1 =2^{-1/3}$ & $a_3=q^{-3}2^{1/6} \mapsto c_3=q^{-3}2^{-1/3}$ \\
        & \tikz[scale=0.3]{\draw (2,1)--(2,3); \draw (1,2)--(3,2); \draw[color=wongblue, very thick] (1,2)--(2,2)--(2,1); } $\mapsto$ \tikz[scale=0.3]{\draw (2,1)--(2,3); \draw (1,2)--(3,2);} & $c_1=2^{-1/3} \mapsto a_1 =2^{-1/3}$ & $c_3=q^{-3}2^{-1/3} \mapsto a_3=q^{-3}2^{1/6}$ \\
        \hline
        \multirow{2}{*}{$(i, j)$} & \tikz[scale=0.3]{\draw (2,1)--(2,3); \draw (1,2)--(3,2);} $\mapsto$ \tikz[scale=0.3]{\draw (2,1)--(2,3); \draw (1,2)--(3,2); \draw[color=wongblue, very thick] (2,3)--(2,2)--(3,2);} & $a_3=q^{-3} 2^{1/6} \mapsto c_3 =q^{-3} 2^{-1/3}$ & $a_1=2^{-1/3} \mapsto c_1=2^{-1/3}$ \\
        & \tikz[scale=0.3]{\draw (2,1)--(2,3); \draw (1,2)--(3,2); \draw[color=wongblue, very thick] (1,2)--(2,2)--(2,1);} $\mapsto$ \tikz[scale=0.3]{\draw (2,1)--(2,3); \draw (1,2)--(3,2); \draw[color=wongblue, very thick] (1,2)--(2,2)--(2,1);\draw[color=wongblue, very thick] (2,3)--(2,2)--(3,2);} & $c_3=q^{-3} 2^{-1/3} \mapsto a_3 =q^{-3} 2^{1/6}$ & $c_1=2^{-1/3} \mapsto a_1=2^{-1/3}$ \\
        \hline
        \multirow{2}{*}{$(i+1, j)$} & \tikz[scale=0.3]{\draw (2,1)--(2,3); \draw (1,2)--(3,2);\draw[color=wongblue, very thick] (2,3)--(2,2)--(3,2);} $\mapsto$ \tikz[scale=0.3]{\draw (2,1)--(2,3); \draw (1,2)--(3,2); \draw[color=wongblue, very thick] (1,2)--(2,2)--(3,2);} & $c_3=q^{-3} 2^{-1/3} \mapsto b_3 =q^{-3} 2^{-1/3}$ & $c_1=2^{-1/3} \mapsto b_1=2^{1/6}$ \\
        & \tikz[scale=0.3]{\draw (2,1)--(2,3); \draw (1,2)--(3,2);\draw[color=wongblue, very thick] (2,1)--(2,2)--(2,3);} $\mapsto$ \tikz[scale=0.3]{\draw (2,1)--(2,3); \draw (1,2)--(3,2); \draw[color=wongblue, very thick] (2,1)--(2,2)--(1,2);} & $b_3=q^{-3} 2^{-1/3} \mapsto c_3 =q^{-3} 2^{-1/3}$ & $b_1=2^{1/6} \mapsto c_1=2^{-1/3}$ \\
        \hline
    \end{tabular}
    \caption{All possible changes of the weight at a vertex.} \label{tab:changes-of-weights}
\end{table}

\section{A surjective map from mixed \texorpdfstring{$6$}{6}V configurations to triple-free GT patterns} \label{sec:map-mixed-6V-to-amtri}
In this section, we construct a surjective map
\[
\psi : \{\text{mixed $6$V configurations on $\mathcal{M}_{\mathbf{k}}$}\} \longrightarrow \{\text{triple-free GT patterns with bottom row~$\mathbf{k}$}\}.
\]
We then show that under this map, the weight of a triple-free GT pattern is equal to the total weight of all configurations mapped to it.
This construction provides the final combinatorial link in our chain of probabilistic bijections, connecting the weighted enumeration of mixed $6$V configurations to that of triple-free GT patterns.

The reader may find it helpful to compare the following definition of $\psi$ with the example shown in \cref{fig:example-psi}.
\begin{defi} \label{defi:psi}
    Let $\mathbf{k} = (k_1, k_2, \ldots, k_n)$ be a strictly increasing sequence of positive integers.
    We define a map~$\psi$ from the set of mixed $6$V configurations on the domain~$\mathcal{M}_\mathbf{k}$ to the set of triple-free GT patterns with bottom row~$\mathbf{k}$ as the composition of the two maps~$\psi_1,\psi_2$ defined as follows.

    Step 1. The map~$\psi_1$ assigns to each mixed $6$V configuration~$x$ a monotone triangle~$T=(T_{i,j})_{1 \leq i \leq n, 1 \leq j \leq i}$, that is, a triangular array with strictly increasing rows satisfying the interlacing conditions~$T_{i+1,j} \leq T_{i,j} \leq T_{i+1,j+1}$.
    Before defining the entries of $T$, we relabel the row positions according to
    \[
        2y-1 \mapsto y \quad (1 \leq y \leq k_n),
        \qquad
        2y \mapsto \overline{y} \quad (1 \leq y \leq k_n-1),
    \]
    and use the order $1 < \overline{1} < 2 < \overline{2} < \cdots < k_n-1 < \overline{k_n-1} < k_n$.
    The $i$-th row of $T$ is then obtained by recording, after this relabelling, the rows in which the lattice paths enter column~$n-i+1$ from the left.

    Step 2. The map~$\psi_2$ assigns to each monotone triangle array~$T$ obtained by $\psi_1$ a triple-free GT pattern~$T'$ obtained by removing all bars.
    More explicitly, for $T=(T_{i,j})_{1 \leq i \leq n, 1 \leq j \leq i}$, define $T'=(T'_{i, j})_{1 \leq i \leq n, 1 \leq j \leq i}$ by
    \begin{equation*}
        T'_{i, j} = \begin{cases}
            y & \text{if}\quad T_{i,j} = y \quad \text{for some $y \in \{1, 2, \ldots, k_n\}$}; \\
            y & \text{if}\quad T_{i,j} = \overline{y} \quad  \text{for some $y \in \{1, 2, \ldots, k_n\}$}.
        \end{cases}
    \end{equation*}

    Then, $\psi$ is defined by $\psi = \psi_2 \circ \psi_1$.
\end{defi}

Note that the composition~$\psi_1$ gives a bijection from the set of mixed $6$V configurations on the domain~$\mathcal{M}_\mathbf{k}$ to the set of monotone triangles with entries in $\{1 < \overline{1} < 2 < \overline{2} < \cdots < k_n-1 < \overline{k_n-1} < k_n\}$ and with bottom row~$k$, while $\psi_2$ may map two distinct monotone triangles to the same triple-free GT pattern.

\begin{figure}[htb]
    \centering
    $x = $
    \scalebox{0.8}{
    \begin{tikzpicture}[scale=0.5,baseline=(current bounding box.center)]
        \draw (1,1) -- (6,1);
        \draw (1,2) -- (6,2);
        \draw (1,3) -- (6,3);
        \draw (1,4) -- (6,4);
        \draw (1,5) -- (6,5);
        \draw (1,6) -- (6,6);
        \draw (1,7) -- (6,7);
        \draw (1,8) -- (6,8);
        \draw (1,9) -- (6,9);
        \draw (1,10) -- (6,10);
        \draw (1,11) -- (6,11);
        \draw (1,12) -- (6,12);
        \draw (1,13) -- (6,13);
        \draw (1,14) -- (6,14);
        \draw (1,15) -- (6,15);

        \draw (1,1) -- (1,15);
        \draw (2,1) -- (2,15);
        \draw (3,1) -- (3,15);
        \draw (4,1) -- (4,15);
        \draw (5,1) -- (5,15);
        \draw (6,1) -- (6,15);

        \draw[color=wongblue, ultra thick] (0.5,3) -- (1,3) -- (1,0.5);
        \draw[color=wongblue, ultra thick] (0.5,5) -- (1,5) -- (1,4) -- (2,4) -- (2,0.5);
        \draw[color=wongblue, ultra thick] (0.5,9) -- (1,9) -- (1,8) -- (2,8) -- (2,7) -- (3,7) -- (3,0.5);
        \draw[color=wongblue, ultra thick] (0.5,11)--(1,11)--(1,10)--(2,10)--(2,8)--(3,8)--(3,7)--(4,7)--(4,0.5);
        \draw[color=wongblue, ultra thick] (0.5,13)--(3,13)--(3,10)--(4,10)--(4,8)--(5,8)--(5,0.5);
        \draw[color=wongblue, ultra thick] (0.5,15)--(2,15)--(2,14)--(3,14)--(3,13)--(4,13)--(4,12)--(5,12)--(5,9)--(6,9)--(6,0.5);

        \node at (7, 16) {row};
        \foreach \y in {1,...,7}{
            \node at (7, \y+\y) {$\overline{\the\numexpr\y}$};
        }
        \foreach \y in {1,...,8}{
            \node at (7, \y+\y-1) {$\the\numexpr\y$};
        }
    \end{tikzpicture}
    $\overset{\psi_1}{\mapsto}$
    $T = \begin{array}{ccccccccccc}
    & & & & & 5 & & & & & \\
    & & & & \overline{4} & & \overline{6} & & & & \\
    & & & 4 & & \overline{5} & & 7 & & & \\
    & & 4 & & \overline{4} & & 7 & & \overline{7} & & \\
      & \overline{2} & & \overline{4} & & \overline{5} & & 7 & & 8 & \\
    2 &   & 3 &   & 5 &   & 6 &   & 7 &   & 8
    \end{array}$
    $\overset{\psi_2}{\mapsto}$
    $T' = \begin{array}{ccccccccccc}
    & & & & & 5 & & & & & \\
    & & & & 4 & & 6 & & & & \\
    & & & 4 & & 5 & & 7 & & & \\
    & & 4 & & 4 & & 7 & & 7 & & \\
      & 2 & & 4 & & 5 & & 7 & & 8 & \\
    2 &   & 3 &   & 5 &   & 6 &   & 7 &   & 8
    \end{array}$
    }
    \caption{An example of the map~$\psi$ with $\mathbf{k} = (2,3,5,7,8)$.} \label{fig:example-psi}
\end{figure}
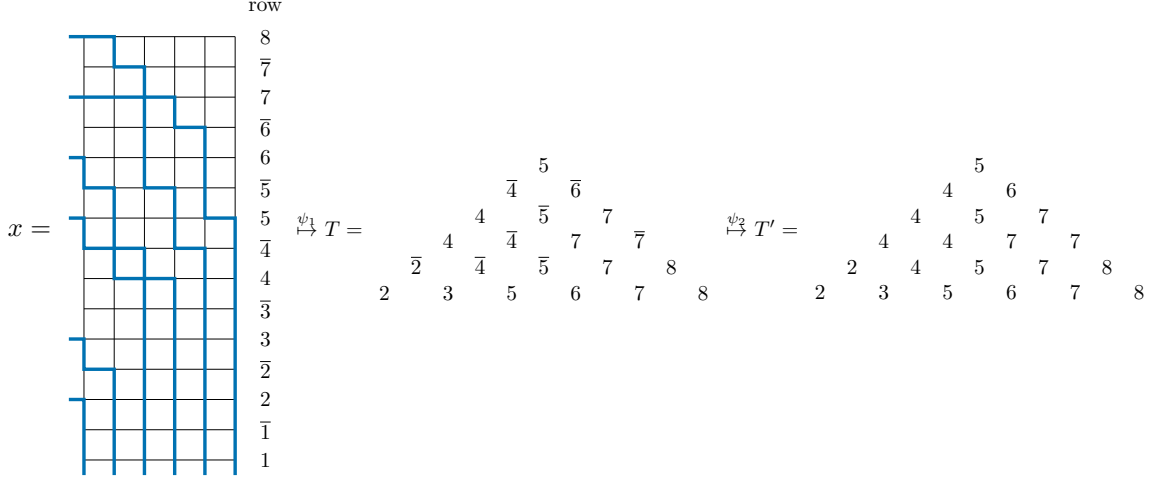

The goal of this section is to show the following theorem.

\begin{theo}\label{theo:psi-is-correct-map}
    Let $\mathbf{k} = (k_1, k_2, \ldots, k_n)$ be a strictly increasing sequence of positive integers.
    For every triple-free GT pattern~$T'$ with bottom row~$k$, the following holds:
    \begin{equation*}
        \sum_x 2^{\ic(x)} = 2^{-n} \omega_{\FSA}(T'),
    \end{equation*}
    where $x$ runs over all mixed $6$V configurations on $\mathcal{M}_\mathbf{k}$ such that $\psi(x)=T'$.
\end{theo}
By \cref{lemm:proper-weight-distribution-is-prob-bij}, this theorem immediately implies the following probabilistic bijection between mixed $6$V configurations and triple-free GT patterns.
\begin{coro} \label{coro:mixed6V-AGTP-prob-bij}
    Let $\mathbf{k} = (k_1, k_2, \ldots, k_n)$ be a strictly increasing sequence of positive integers.
    There exists a probabilistic bijection between the set of mixed $6$V configurations on the domain~$\mathcal{M}_\mathbf{k}$ with the weight function~$2^{\ic(\cdot)}$ and the set of triple-free GT patterns with bottom row~$\mathbf{k}$ with the weight function~$2^{-n} \omega_{\FSA}(\cdot)$.
\end{coro}

The proof of \cref{theo:psi-is-correct-map} proceeds in two steps:
\begin{enumerate}
    \item Express the variant inversion number in terms of monotone triangles~$T$ (\cref{prop:ic-in-alpha-and-beta}).
    \item Compute the total weight of all monotone triangles~$T$ mapping to a given triple-free GT pattern~$T'$ (\cref{prop:omega-in-ic}).
\end{enumerate}

\subsection{From mixed \texorpdfstring{$6$}{6}V configurations to monotone triangles}
In the rest of the paper, $\#\left\{\begin{array}{cc} & t \\ t & \end{array}\ \text{in}\ T\right\}$ denotes the number of pairs~$(i, j)$ such that $T_{i, j} = T_{i+1, j}$.
The other similar notations are interpreted in the same manner as well.

\begin{prop} \label{prop:ic-in-alpha-and-beta}
    Let $x$ be a mixed $6$V configuration on the domain~$\mathcal{M}_\mathbf{k}$, and let $T = \psi_1(x)$.
    Let
    \begin{align*}
        \overline{\alpha}(T) &= \#\Set{(i, j)}{1 \leq j \leq i \leq n-1, T_{i,j} \in \{1, 2, \ldots, k_n\}, T_{i, j}=T_{i+1,j}} =  \#\left\{\begin{array}{cc} & \overline{t} \\ \overline{t} & \end{array}\ \text{in}\ T \right\}, \\
        \beta(T) &= \#\Set{(i, j)}{1 \leq j \leq i \leq n-1, T_{i,j} \in \{\overline{1}, \overline{2}, \ldots, \overline{k_n-1}\}, T_{i, j}=T_{i+1,j+1}} = \#\left\{\begin{array}{cc}t & \\ & t \end{array}\ \text{in}\ T\right\}.
    \end{align*}
    Then, we have
    \begin{equation}
        \mathcal{N}_{(1)}^{\even} = \overline{\alpha}(T) \quad\text{and}\quad \mathcal{N}_{(3)}^{\odd}(x) = \beta(T),
    \end{equation}
    and thus, we obtain
    \begin{equation}
        \ic(x) = \overline{\alpha}(T) + \beta(T) = \#\left\{\begin{array}{cc} & \overline{t} \\ \overline{t} & \end{array}\ \text{in}\ T \right\} + \#\left\{\begin{array}{cc}t & \\ & t \end{array}\ \text{in}\ T\right\}. \label{eq:ic-in-alpha-and-beta}
    \end{equation}
\end{prop}

\begin{proof}
    We only prove $\mathcal{N}_{(1)}^{\even} = \overline{\alpha}(T)$ since $\mathcal{N}_{(3)}^{\odd}(x) = \beta(T)$ follows from a similar argument.

    Assume that we have $T_{i, j}=T_{i+1,j} \eqqcolon \overline{t}$ for some $i, j$ and $t \in [k_n-1]$.
    In terms of non-intersecting lattice paths, this means that the ($n-i+j$)-th path enters column~$n-i+1$ at $y=2t$, and that the ($n-i+j-1$)-th path enters column~$n-i$ at $y=2t$.
    Thus, the local configuration at vertex~$(n-i, 2t)$ must be type-$1$.

    It is also straightforward to see that given a vertex whose local configuration is type~$1$, the corresponding two entries in $T$ form the pattern~$\begin{array}{cc} & \overline{t} \\ \overline{t} & \end{array}$.
    Hence, we obtain $\mathcal{N}_{(1)}^{\even} = \overline{\alpha}(T)$.
\end{proof}

\cref{prop:ic-in-alpha-and-beta} justifies the following definition.

\begin{defi}
    Let $T$ be a monotone triangle with entries in $\{1 < \overline{1} < 2 < \overline{2} < 3 < \overline{3} < \cdots\}$.
    We define $\ic(T)$ by $\ic(T) = \overline{\alpha}(T) + \beta(T) = \#\left\{\begin{array}{cc} & \overline{t} \\ \overline{t} & \end{array}\ \text{in}\ T \right\} + \#\left\{\begin{array}{cc}t & \\ & t \end{array}\ \text{in}\ T\right\}$.
\end{defi}

\begin{exam}
    Take the configuration~$x$ as in \cref{exam:inv-crossing}.
    Then $T=\psi_1(x)$ is $\begin{array}{ccccccc} & & & \overline{2} & & & \\
    & & 2 & & \overline{2} & & \\
    & 2 & & \overline{2} & & 4 & \\
    1 & & 2 & & 3 & & 4\end{array}$.
    Thus, $\overline{\alpha}(T) = 1, \beta(T) = 2$.
    Hence, $\ic(T) = \overline{\alpha}(T) + \beta(T) = 3 = \ic(x)$ as expected from \cref{prop:ic-in-alpha-and-beta}.
\end{exam}

\subsection{From monotone triangles to triple-free GT patterns}
Our goal in this subsection is to show the following proposition.

\begin{prop} \label{prop:omega-in-ic}
    Let $T'$ be a triple-free GT pattern with bottom row~$\mathbf{k} = (k_1, k_2, \ldots, k_n)$, strictly increasing.
    Then, we have
    \begin{equation}
        2^{-n} \omega_{\FSA}(T') = \sum_{T} 2^{\ic(T)}, \label{eq:FSA-and-ic}
    \end{equation}
    where $T$ runs over all the monotone triangles with entries in $\{1 < \overline{1} < 2 < \overline{2} < \cdots < k_n-1 < \overline{k_n-1} < k_n\}$ and with bottom row~$\mathbf{k}$ such that $\psi_2(T) = T'$.
\end{prop}

\begin{exam}
    Let $T' = \begin{array}{ccccccc} & & & 2 & & & \\
    & & 2 & & 3 & & \\
    & 2 & & 3 & & 3 & \\
    1 & & 2 & & 3 & & 4\end{array}$.
    Let us verify that both sides of \cref{eq:FSA-and-ic} agree for this $T'$.
    The LHS is $2^{-4} \omega_{\FSA}(T') = 2^5$.
    The list of all $T$ from which $T'$ can be obtained by removing all bars is as follows:
    \begin{center}
        $\begin{array}{ccccccc} & & & 2 & & & \\
        & & 2 & & 3 & & \\
        & 2 & & 3 & & \overline{3} & \\
        1 & & 2 & & 3 & & 4\end{array}$,
        $\begin{array}{ccccccc} & & & \overline{2} & & & \\
        & & 2 & & 3 & & \\
        & 2 & & 3 & & \overline{3} & \\
        1 & & 2 & & 3 & & 4\end{array}$,
        $\begin{array}{ccccccc} & & & \overline{2} & & & \\
        & & \overline{2} & & 3 & & \\
        & 2 & & 3 & & \overline{3} & \\
        1 & & 2 & & 3 & & 4\end{array}$,

        \vspace{15pt}

        $\begin{array}{ccccccc} & & & 2 & & & \\
        & & 2 & & \overline{3}& & \\
        & 2 & & 3 & & \overline{3} & \\
        1 & & 2 & & 3 & & 4\end{array}$,
        $\begin{array}{ccccccc} & & & \overline{2} & & & \\
        & & 2 & & \overline{3} & & \\
        & 2 & & 3 & & \overline{3} & \\
        1 & & 2 & & 3 & & 4\end{array}$,
        $\begin{array}{ccccccc} & & & \overline{2}& & & \\
        & & \overline{2} & & \overline{3} & & \\
        & 2 & & 3 & & \overline{3} & \\
        1 & & 2 & & 3 & & 4\end{array}$.
    \end{center}
    The statistic $\ic(T)$ for each $T$, listed from top to bottom and left to right, is $2, 2, 3, 2, 2, 3$.
    Thus, the RHS is $2^2+2^2+2^3+2^2+2^2+2^3 = 2^5$.
    Hence, both sides of \cref{eq:FSA-and-ic} agree.
\end{exam}

We show \cref{prop:omega-in-ic} by decomposing GT patterns into smaller parts.

\begin{defi}
    Let $T = (T_{i,j})_{1 \leq i \leq n, 1 \leq j \leq i}$ be a triangular array with entries in $\{1 < \overline{1} < 2 < \overline{2} < 3 < \overline{3} < \cdots\}$.
    We call a subarray~$C \subseteq T$ a \myemph{connected block} of $T$ if it satisfies both of the following two conditions.
    \begin{enumerate}
        \item There exists an integer~$t$ such that every entry in $C$ equals $t$ or $\overline{t}$.
        \item Regard $C$ as the undirected graph where the vertices are all entries of $C$ and two entries are connected by an edge if and only if the two entries are diagonally adjacent in the triangular array. Then, $C$ is a connected graph.
    \end{enumerate}
    A connected block~$C$ is called \myemph{maximal} if there is no connected block that strictly contains $C$.
    Denote by $\mc(T)$ the set of all maximal connected blocks of $T$.

    For a connected block~$C$, we define $\overline{\alpha}(C)$ and $\beta(C)$ by the numbers of occurrences of the patterns~$\begin{array}{cc} & \overline{t} \\ \overline{t} & \end{array}$, $\begin{array}{cc}t & \\ & t \end{array}$ in $C$, respectively.
    Furthermore, define $\ic(C)$ by $\ic(C) = \overline{\alpha}(C) + \beta(C)$.
\end{defi}

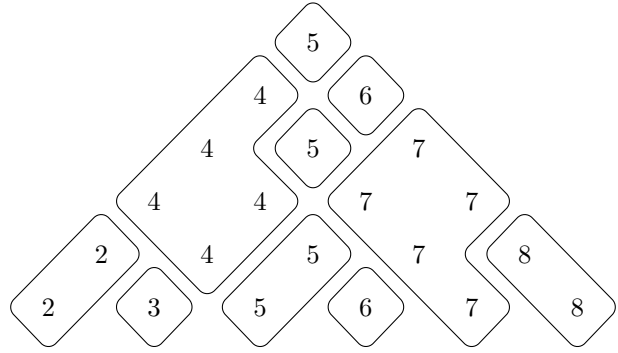
\begin{figure}[htb]
    \centering
    \begin{tikzpicture}[scale=0.7]
        \node[outer sep=10pt] (0-0) at (0,0) {$2$};
        \node[outer sep=10pt] (2-0) at (2,0) {$3$};
        \node[outer sep=10pt] (4-0) at (4,0) {$5$};
        \node[outer sep=10pt] (6-0) at (6,0) {$6$};
        \node[outer sep=10pt] (8-0) at (8,0) {$7$};
        \node[outer sep=10pt] (10-0) at (10,0) {$8$};

        \node[outer sep=10pt] (1-1) at (1,1) {$2$};
        \node[outer sep=10pt] (3-1) at (3,1) {$4$};
        \node[outer sep=10pt] (5-1) at (5,1) {$5$};
        \node[outer sep=10pt] (7-1) at (7,1) {$7$};
        \node[outer sep=10pt] (9-1) at (9,1) {$8$};

        \node[outer sep=10pt] (2-2) at (2,2) {$4$};
        \node[outer sep=10pt] (4-2) at (4,2) {$4$};
        \node[outer sep=10pt] (6-2) at (6,2) {$7$};
        \node[outer sep=10pt] (8-2) at (8,2) {$7$};

        \node[outer sep=10pt] (3-3) at (3,3) {$4$};
        \node[outer sep=10pt] (5-3) at (5,3) {$5$};
        \node[outer sep=10pt] (7-3) at (7,3) {$7$};

        \node[outer sep=10pt] (4-4) at (4,4) {$4$};
        \node[outer sep=10pt] (6-4) at (6,4) {$6$};

        \node[outer sep=10pt] (5-5) at (5,5) {$5$};

        \draw[rounded corners] (0-0.west) -- (0-0.south) -- (1-1.east) -- (1-1.north) -- cycle;
        \draw[rounded corners] (2-0.south) -- (2-0.east) -- (2-0.north) -- (2-0.west) -- cycle;
        \draw[rounded corners] (4-0.west) -- (4-0.south) -- (5-1.east) -- (5-1.north) -- cycle;
        \draw[rounded corners] (6-0.south) -- (6-0.east) -- (6-0.north) -- (6-0.west) -- cycle;
        \draw[rounded corners] (8-0.south) -- (8-0.east) -- (7-1.east) -- (8-2.east) -- (7-3.north) -- (6-2.west) -- cycle;
        \draw[rounded corners] (10-0.south) -- (10-0.east) -- (9-1.north) -- (9-1.west) -- cycle;

        \draw[rounded corners] (3-1.south) -- (4-2.east) -- (3-3.east) -- (4-4.east) -- (4-4.north) -- (2-2.west) -- cycle;

        \draw[rounded corners] (5-3.south) -- (5-3.east) -- (5-3.north) -- (5-3.west) -- cycle;
        \draw[rounded corners] (6-4.south) -- (6-4.east) -- (6-4.north) -- (6-4.west) -- cycle;
        \draw[rounded corners] (5-5.south) -- (5-5.east) -- (5-5.north) -- (5-5.west) -- cycle;
    \end{tikzpicture}
    \caption{The maximal connected blocks of a triple-free GT pattern.}
    \label{fig:example-of-maximal-connected blocks}
\end{figure}

The following lemma shows that the variant inversion number~$\ic(T)$ naturally decomposes into those of the maximal connected blocks.
\begin{lemm} \label{lemm:ic-into-smaller-parts}
    Let $T$ be a monotone triangle with entries in $\{1 < \overline{1} <2 < \overline{2} < \cdots < k_n-1 < \overline{k_n-1} < k_n\}$ and with strictly increasing bottom row~$\mathbf{k} = (k_1, k_2, \ldots, k_n)$.
    Then, we have
    \begin{equation*}
        \ic(T) = \sum_{C \in \mc(T)} \ic(C).
    \end{equation*}
\end{lemm}

\begin{proof}
    Two diagonally adjacent elements in distinct maximal connected blocks do not contribute to $\ic(T)$. Thus, we can calculate $\ic(T)$ by calculating $\ic(C)$ for each $C \in \mc(T)$ and taking the sum over them.
\end{proof}

We also define $\omega_{\FSA}(\cdot)$ for maximal connected blocks.

\begin{defi}
    Let $T'$ be a triple-free GT pattern with entries in $\{1 < 2 < \cdots < k_n-1 < k_n\}$ and with strictly increasing bottom row~$\mathbf{k} = (k_1, k_2, \ldots, k_n)$.
    Let $C' \in \mc(T')$.
    Define the weight~$\omega_{\FSA}(C')$ of $C'$ by
    \begin{align*}
        \omega_{\FSA}(C') = 2^{\#C' - \#\{t\quad t\ \text{in}\ C'\}},
    \end{align*}
    where $\#C'$ is the number of entries in $C'$ and $\#\{t\quad t\ \text{in}\ C'\}$ counts the number of equal-entry pairs in the rows of $C'$.
\end{defi}

The next lemma shows that the weight~$\omega_{\FSA}(T')$ naturally decomposes into those of the maximal connected blocks, analogously to \cref{lemm:ic-into-smaller-parts}.
\begin{lemm} \label{lemm:omega-into-smaller-parts}
    Let $T'$ be a triple-free GT pattern with entries in $\{1 < 2 < \cdots < k_n-1 < k_n\}$ and with strictly increasing bottom row~$\mathbf{k} = (k_1, k_2, \ldots, k_n)$.
    Then, we have
    \begin{equation*}
        \omega_{\FSA}(T') = \prod_{C' \in \mc(T')} \omega_{\FSA}(C').
    \end{equation*}
\end{lemm}
This lemma is immediate to show.

We introduce the following terms to describe the structure of connected blocks~$C$; these terms will be used in the proof of \cref{prop:omega-in-ic-in-small-parts} below.
We define the \myemph{width} of a row of $C$ to be the number of entries of $C$ in that row.
A maximal sequence of consecutive rows of $C$ of width~$1$ is called a \myemph{zigzag part}.
A row of $C$ of width~$2$ is called a \myemph{joint part}.
It is easy to see that, in general, a connected block consists of alternating zigzag and joint parts, starting and ending with zigzag parts.

In the following, for a subarray~$C \subseteq T$, let $\psi_2(C)$ denote the subarray obtained from $C$ by removing all bars, and define $b_T(C) \coloneqq [\text{$C$ intersects the bottom row of $T$}]$.
Analogously, for $C' \subseteq T'$, define $b_{T'}(C')$.
\begin{prop} \label{prop:omega-in-ic-in-small-parts}
    Let $T'$ be a triple-free GT pattern with entries in $\{1 < 2 < \cdots < k_n-1 < k_n\}$ and with strictly increasing bottom row~$\mathbf{k} = (k_1, k_2, \ldots, k_n)$.
    Let $C' \in \mc(T')$.
    Then, we have
    \begin{equation}
        2^{-b_{T'}(C')} \omega_{\FSA}(C') = \sum_{C \in \psi_2^{-1}(C')} 2^{\ic(C)}, \label{eq:ic(C)-to-ic(F)}
    \end{equation}
    where the sum is restricted to those $C$ whose bottom-row entry, if present, is unbarred.
\end{prop}

\begin{proof}
    \renewcommand{\qedsymbol}{}
    We give the argument in the case where $C'$ has no entry in the bottom row of $T'$.
    The case where $C'$ has a bottom-row entry is identical, except that one restricts to preimages whose bottom-row entry is unbarred.

    Let $C' \in \mc(T')$, and let $C \in \psi_2^{-1}(C')$.
    Clearly, there are at most two equal entries in $C'$ in each row.
    If a row of $C'$ contains two entries, then $C$ is uniquely determined at that row.
    Namely, if a row of $C'$ contains the pattern~$t\qquad t$, the corresponding entries in $C$ must be the pattern~$t\qquad \overline{t}$,

    Moreover, it turns out that the contribution to $\ic(C)$ by such rows does not depend on the choice of $C$.
    To show that, assume that $C$ contains the pattern~$t\qquad t$ for some $t$ in a row and there is another $t$ above the two equal entries, say, $\begin{array}{ccc} & t & \\ t & & t \end{array}$.
    Then, $C$ must be either $\begin{array}{ccc} & t & \\ t & & \overline{t} \end{array}$ or $\begin{array}{ccc} & \overline{t} & \\ t & & \overline{t} \end{array}$ in those two rows.
    However, neither of the two triangles contains any of the two patterns~$\begin{array}{cc} & \overline{t} \\ \overline{t} & \end{array}$ and $\begin{array}{cc} t & \\ & t\end{array}$.
    Thus, the contribution to $\ic(C)$ in those two rows is $0$ either way.
    Similarly, we can show that if there is another $t$ below the equal entries, the two rows always contribute to $\ic(C)$ by $+1$ independently of the choice of $C$.

    Denote by $Z_1, Z_2, \ldots, Z_{\ell}$ (resp. $Z'_1, Z'_2, \ldots, Z'_{\ell}$) the zigzag parts of $C$ (resp. $C'$) from bottom to top, and denote by $J_1, J_2, \ldots, J_{\ell}$ (resp. $J'_1, J'_2, \ldots, J'_{\ell}$) the joint parts of $C'$ from bottom to top.
    Since the zigzag parts and joint parts alternate, $C$ and $C'$ can be schematically represented as follows using $Z_i, J_i, Z'_i$, and $J'_i$.
    \begin{center}
        \begin{tikzpicture}[scale=0.6]
            \node at (-2,5.5) {$C:$};

            \node (F1) at (1,0) {$Z_1$};
            \node (x1) at (0,1) {$t$};
            \node (xb1) at (2,1) {$\overline{t}$};
            \node (F2) at (1,2) {$Z_2$};
            \node (x2) at (0,3) {$t$};
            \node (xb2) at (2,3) {$\overline{t}$};

            \node (A) at (1,4) {};
            \node (B) at (1,5) {};

            \node (Fl-1) at (1,6) {$Z_{\ell-1}$};
            \node (xl-1) at (0,7) {$t$};
            \node (xbl-1) at (2,7) {$\overline{t}$};
            \node (Fl) at (1,8) {$Z_{\ell}$};

            \draw (x1) -- (F1);
            \draw (xb1) -- (F1);
            \draw (F2) -- (x1);
            \draw (F2) -- (xb1);
            \draw (x2) -- (F2);
            \draw (xb2) -- (F2);
            \draw (x2) -- (0.5,3.5);
            \draw (xb2) -- (1.5,3.5);

            \draw[dotted,very thick] (A) -- (B);

            \draw (Fl-1) -- (0.3,5.3);
            \draw (Fl-1) -- (1.7,5.3);
            \draw (xl-1) -- (Fl-1);
            \draw (xbl-1) -- (Fl-1);
            \draw (Fl) -- (xl-1);
            \draw (Fl) -- (xbl-1);
        \end{tikzpicture}
        \hspace{20mm}
        \begin{tikzpicture}[scale=0.6]
            \node at (-2,5.5) {$C:$};

            \node (F1) at (1,0) {$Z'_1$};
            \node (x1) at (0,1) {$t$};
            \node (xb1) at (2,1) {$t$};
            \node (F2) at (1,2) {$Z'_2$};
            \node (x2) at (0,3) {$t$};
            \node (xb2) at (2,3) {$t$};

            \node (A) at (1,4) {};
            \node (B) at (1,5) {};

            \node (Fl-1) at (1,6) {$Z'_{\ell-1}$};
            \node (xl-1) at (0,7) {$t$};
            \node (xbl-1) at (2,7) {$t$};
            \node (Fl) at (1,8) {$Z'_{\ell}$};

            \draw (x1) -- (F1);
            \draw (xb1) -- (F1);
            \draw (F2) -- (x1);
            \draw (F2) -- (xb1);
            \draw (x2) -- (F2);
            \draw (xb2) -- (F2);
            \draw (x2) -- (0.5,3.5);
            \draw (xb2) -- (1.5,3.5);

            \draw[dotted,very thick] (A) -- (B);

            \draw (Fl-1) -- (0.3,5.3);
            \draw (Fl-1) -- (1.7,5.3);
            \draw (xl-1) -- (Fl-1);
            \draw (xbl-1) -- (Fl-1);
            \draw (Fl) -- (xl-1);
            \draw (Fl) -- (xbl-1);
        \end{tikzpicture}
    \end{center}
    For distinct $i$ and $j$, the configurations in $Z_i$ and $Z_j$ can be determined independently since $Z_i$ and $Z_j$ are separated by one or more joint parts.
    Thus, the RHS of \cref{eq:ic(C)-to-ic(F)} can be transformed as follows:
    \begin{equation}
        \sum_{C} 2^{\ic(C)} = \left(\prod_{i=1}^{\ell} \sum_{Z_i \in \psi_2^{-1}(Z'_i)} 2^{\ic(Z_i)} \right) \times 2^{\ell-1}. \label{eq:ic-of-C-in-ic-of-Z}
    \end{equation}
    Here, the factor~$2^{\ell-1}$ comes from the argument about joint parts above.

    Hence, we want to calculate the sum~$\sum_{F_i\text{ preimage of }Z'_i} 2^{\ic(Z_i)}$.
    We defer this computation to the next lemma.
    After proving the lemma, we return to the proof of this proposition and complete it.
\end{proof}

\begin{lemm}\label{lemm:ic-of-zigzag}
    Let $Z'$ be a zigzag part of a triple-free GT pattern.
    Then, we have
    \begin{equation}
        \sum_{Z \in \psi_2^{-1}(Z')} 2^{\ic(Z)} = 2^{\#Z'}. \label{eq:ic-of-zigzag}
    \end{equation}
\end{lemm}

Before proving this lemma, we express our objects in terms of order ideals of certain posets.

\begin{defi}
    A finite poset is called a \myemph{fence} if its Hasse diagram is a path graph.
    In other words, a fence is a poset whose Hasse diagram is a connected graph such that there are exactly two vertices with degree~$1$ and all other vertices have degree~$2$.
\end{defi}

Clearly, the underlying shape of $Z'$ can be identified, after a $90^\circ$ counterclockwise rotation, with the Hasse diagram of a fence $F$, and each subarray~$Z$ corresponds to an order ideal of $F$.
In the following, an element in an order ideal is represented by a blue point, while a black point represents an element in its complement. See \cref{fig:counterclockwise-rotation-example} for an example.

After a $90^\circ$ counterclockwise rotation, $\overline{\alpha}(Z)$ and $\beta(Z)$ count the numbers of occurrences of the patterns~\tikz[scale=0.3]{\node[circle,fill=wongblack,minimum size = 0.2cm, inner sep=0pt] (a) at (1,0) {}; \node[circle,fill=wongblack,minimum size = 0.2cm, inner sep=0pt] (b) at (0,1) {}; \draw (a)--(b);} and \tikz[scale=0.3]{\node[circle,fill=wongblue,minimum size = 0.2cm, inner sep=0pt] (a) at (0,0) {}; \node[circle,fill=wongblue,minimum size = 0.2cm, inner sep=0pt] (b) at (1,1) {}; \draw (a)--(b);}, respectively.
Thus, for an order ideal~$I$ of the fence~$F$, we define $\ic(I)$ as the total number of occurrences of these two patterns.
For example, the order ideal~$I$ in \cref{fig:counterclockwise-rotation-example} has $\ic(I) = 1+1=2$.
By this translation, \cref{eq:ic-of-zigzag} can be expressed in terms of order ideals of a fence as follows:
\begin{equation}
    \sum_{I \in \mathcal{J}(F)} 2^{\ic(I)} = 2^{\#F}, \label{eq:ic-in-order-ideals}
\end{equation}
where $\mathcal{J}(F)$ is the set of order ideals of $F$.

\begin{figure}[bht]
    \centering
    \begin{tikzpicture}[scale=0.5]
        \begin{scope}
            \node at (1, 1) {$t$};
            \node at (0, 2) {$t$};
            \node at (1, 3) {$\overline{t}$};
            \node at (2, 4) {$\overline{t}$};
            \node at (1, 5) {$t$};
            \node at (2, 6) {$\overline{t}$};
        \end{scope}
        \draw[<->] (3,3)--(6,3);
        \begin{scope}[shift={(6,3)}]
            \node[circle,fill=wongblack,minimum size = 0.2cm, inner sep=0pt] (a) at (1,1) {};
            \node[circle,fill=wongblue,minimum size = 0.2cm, inner sep=0pt] (b) at (2,0) {};
            \node[circle,fill=wongblack,minimum size = 0.2cm, inner sep=0pt] (c) at (3,1) {};
            \node[circle,fill=wongblack,minimum size = 0.2cm, inner sep=0pt] (d) at (4,0) {};
            \node[circle,fill=wongblue,minimum size = 0.2cm, inner sep=0pt] (e) at (5,-1) {};
            \node[circle,fill=wongblue,minimum size = 0.2cm, inner sep=0pt] (f) at (6,0) {};
            \draw (a)--(b)--(c)--(d)--(e)--(f);
        \end{scope}
    \end{tikzpicture}
    \caption{Left: a configuration of a subarray. Right: the corresponding order ideal of the fence obtained after a $90^\circ$ counterclockwise rotation. Blue vertices belong to the order ideal, while black vertices belong to its complement.} \label{fig:counterclockwise-rotation-example}
\end{figure}
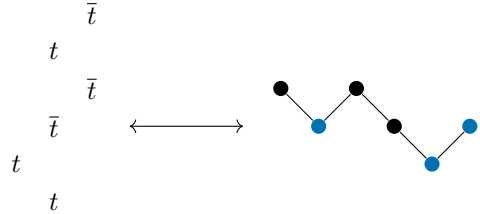

\begin{proof}[Proof of \cref{lemm:ic-of-zigzag}]
    Let $F$ be the fence corresponding to the shape of $Z'$.
    We show \cref{eq:ic-in-order-ideals} by induction on the number of elements~$\#F$.

    \cref{eq:ic-in-order-ideals} is easy to verify when $\#F = 1$.
    Now let $\ell \geq 1$, and suppose that \cref{eq:ic-in-order-ideals} holds for all fences $F$ with size~$\ell$.
    It suffices to show that, for every fence $F$ of size~$\leq\ell$, the fence obtained by attaching one element to the upper right or lower right of $F$ satisfies the equation, since any fence of size $\ell+1$ can be obtained from some fence of size $\ell$ in this way.
    Let $x$ and $y$ denote the second-rightmost and the rightmost elements of $F$, respectively.
    We assume $x \geq y$ holds.
    The other case, where $x \leq y$, can be dealt with similarly, and we omit its proof.
    Let $m$ be the number of elements $w \in F$ such that $w \geq y$.
    Let $F'$ be the fence consisting of the elements that are incomparable to $y$ (or, equivalently, to $x$).
    \Needspace{12\baselineskip}
    The fence~$F$ can be schematically represented as follows.
    \begin{center}
        \begin{tikzpicture}[scale=0.6]
            \node[circle,draw=black,thick,minimum size=18pt] (00) at (0,0) {$y$};
            \node[circle,draw=black,thick,minimum size=18pt] (-11) at (-1,1) {$x$};
            \node[circle,draw=black,thick,minimum size=18pt] (-22) at (-2,2) {};
            \node[circle,draw=black,thick,minimum size=18pt] (-33) at (-3,3) {};
            \node[circle,draw=black,thick,minimum size=18pt] (-44) at (-4,4) {};
            \node (F') at (-5,3) {$F'$};

            \node[inner sep=0pt] (A) at (0.4,0.4) {};
            \node[inner sep=0pt] (B) at (-3.6,4.4) {};

            \draw (F') -- (-44) -- (-33) -- (-22) -- (-11) -- (00);
            \draw[decoration={brace,mirror,amplitude=10pt},decorate] (A) -- node[above right=5pt] {$m$} (B);
        \end{tikzpicture}
    \end{center}

    We now attach one new element~$z$ to the upper right or lower right of $y$.
    First, consider the former case and denote the extended poset with $z$ attached by $F_1$.
    \begin{center}
        \begin{tikzpicture}[scale=0.6]
            \node at (-6.2, 2.5) {$F_1:$};
            \node[circle,draw=black,thick,minimum size=18pt] (11) at (1,1) {$z$};
            \node[circle,draw=black,thick,minimum size=18pt] (00) at (0,0) {$y$};
            \node[circle,draw=black,thick,minimum size=18pt] (-11) at (-1,1) {$x$};
            \node[circle,draw=black,thick,minimum size=18pt] (-22) at (-2,2) {};
            \node[circle,draw=black,thick,minimum size=18pt] (-33) at (-3,3) {};
            \node[circle,draw=black,thick,minimum size=18pt] (-44) at (-4,4) {};
            \node (F') at (-5,3) {$F'$};

            \node[inner sep=0pt] (A) at (0.4,0.4) {};
            \node[inner sep=0pt] (B) at (-3.6,4.4) {};

            \draw (F') -- (-44) -- (-33) -- (-22) -- (-11) -- (00) -- (11);
            \draw[decoration={brace,mirror,amplitude=10pt},decorate] (A) -- node[above right=5pt] {$m$} (B);
        \end{tikzpicture}
    \end{center}
    We divide the set~$\mathcal{J}(F_1)$ of order ideals of $F_1$ into two disjoint sets as follows:
    \begin{equation*}
        \mathcal{K}_1 = \Set{I \in \mathcal{J}(F_1)}{z \notin I} \quad \text{and} \quad \mathcal{K}_2 = \Set{I \in \mathcal{J}(F_1)}{z \in I}.
    \end{equation*}
    For any order ideal in $\mathcal{K}_1$, the patterns~\tikz[scale=0.3]{\node[circle,fill=wongblack,minimum size = 0.2cm, inner sep=0pt] (a) at (1,0) {}; \node[circle,fill=wongblack,minimum size = 0.2cm, inner sep=0pt] (b) at (0,1) {}; \draw (a)--(b);} and \tikz[scale=0.3]{\node[circle,fill=wongblue,minimum size = 0.2cm, inner sep=0pt] (a) at (0,0) {}; \node[circle,fill=wongblue,minimum size = 0.2cm, inner sep=0pt] (b) at (1,1) {}; \draw (a)--(b);} can happen only inside $F$ since \tikz[scale=0.3]{\node[circle,draw=black,thick,minimum size = 0.2cm, inner sep=0pt] (a) at (0,0) {$y$}; \node[circle,draw=black,thick,minimum size = 0.2cm, inner sep=0pt] (b) at (1,1) {$z$}; \draw (a)--(b);} cannot be \tikz[scale=0.3]{\node[circle,fill=wongblue,minimum size = 0.2cm, inner sep=0pt] (a) at (0,0) {}; \node[circle,fill=wongblue,minimum size = 0.2cm, inner sep=0pt] (b) at (1,1) {}; \draw (a)--(b);} when $z$ is colored black.
    Hence, by the induction hypothesis, we have
    \begin{equation}
        \sum_{I \in \mathcal{K}_1} 2^{\ic(I)} = \sum_{I \in \mathcal{J}(F)} 2^{\ic(I)} = 2^{\#F}. \label{eq:K1}
    \end{equation}
    For any order ideal in $\mathcal{K}_2$, $y$ and $z$ are fixed to be blue since when $z$ is blue, $y$ must be blue as well.
    Thus, the pattern~\tikz[scale=0.3]{\node[circle,fill=wongblue,minimum size = 0.2cm, inner sep=0pt] (a) at (0,0) {}; \node[circle,fill=wongblue,minimum size = 0.2cm, inner sep=0pt] (b) at (1,1) {}; \draw (a)--(b);} always happens for the part of $y$ and $z$.
    Also, notice that the part of $x$ and $y$ does not contribute to $\ic(I)$ for any $I \in \mathcal{J}(F')$.
    Hence, applying the induction hypothesis to $F' \setminus \{y, z\} = F \setminus \{y\}$, we get
    \begin{equation}
        \sum_{I \in \mathcal{K}_2} 2^{\ic(I)} = 2^1 \times \sum_{I \in \mathcal{J}(F \setminus \{y\})} 2^{\ic(I)} = 2^{1+\#F-1} = 2^{\#F}. \label{eq:K2}
    \end{equation}
    By combining \cref{eq:K1,eq:K2}, we obtain
    \begin{equation*}
        \sum_{I \in \mathcal{J}(F')} 2^{\ic(I)} = \sum_{I \in \mathcal{K}_1} 2^{\ic(I)} + \sum_{I \in \mathcal{K}_2} 2^{\ic(I)} = 2^{\#F} + 2^{\#F} = 2^{\#F + 1} = 2^{\#F'}.
    \end{equation*}
    The case in which the new element~$z$ is attached to the lower right of $y$ can be handled similarly.
    Therefore, the induction step is complete.
\end{proof}

\begin{rema}
    \cref{eq:ic-in-order-ideals} implies that the generating function
    \begin{equation}
        I_F(q) \coloneqq \sum_{I \in \mathcal{J}(F)} q^{\ic(I)}
    \end{equation}
    specialized at $q = 2$ depends only on the size of the fence~$\#F$.
    However, in general, this is not the case for other values of $q$.

    Regarding this generating function, computational experiments suggest that the number sequence $(a_n)_{n \geq 1}$ defined by
    \begin{equation*}
        a_n = \#\{I_F(1) \mid F\text{ is a fence of size } n\}
    \end{equation*}
    begins with
    \begin{equation*}
        1, 1, 2, 3, 6, 10, 16, 29, 51, 83, 148, 246, \ldots,
    \end{equation*}
    which appears to coincide with \href{https://oeis.org/A294444}{OEIS A294444}.
\end{rema}

\begin{proof}[Continuation of the proof of \cref{prop:omega-in-ic-in-small-parts}]
    We apply \cref{lemm:ic-of-zigzag} to \cref{eq:ic-of-C-in-ic-of-Z}.
    Notice that $\#C' = \sum_{i=1}^{\ell}\#Z'_i + 2(\ell-1)$ and $\ell-1 = \#\{t\quad t\ \text{in}\ C'\}$.
    Then we obtain
    \begin{align*}
        &\left(\prod_{i=1}^{\ell} \sum_{Z_i \in \psi_2^{-1}(Z'_i)} 2^{\ic(Z_i)} \right) \times 2^{\ell-1} =  \left(\prod_{i=1}^{\ell}2^{\#Z'_i}\right) \times 2^{\ell-1} = 2^{\sum_{i=1}^{\ell} \left(\#Z'_i\right)} \times 2^{\ell-1} \\
        &\qquad= 2^{\left(\sum_{i=1}^{\ell} \#Z'_i + 2(\ell-1)\right) - (\ell-1)} = 2^{\#C' - \#\{t\quad t\ \text{in}\ C'\}} = \omega_{\FSA}(C').
    \end{align*}
\end{proof}

\begin{proof}[Proof of \cref{prop:omega-in-ic}]
    Let $T'$ be a triple-free GT pattern and let $C'_1, C'_2, \ldots, C'_\ell$ be its maximal connected blocks.
    Then, we have
    \begin{align*}
        2^{-n} \omega_{\FSA}(T') &= 2^{-n} \prod_{i=1}^{\ell} \omega_{\FSA}(C'_i) \qquad\text{(by \cref{lemm:omega-into-smaller-parts})} \\
        &= \prod_{i=1}^{\ell} \sum_{C_i \in \psi_2^{-1}(C'_i)} 2^{\ic(C_i)} \qquad\text{(by \cref{prop:omega-in-ic-in-small-parts})} \\
        &= \sum_{\substack{C_1, \ldots, C_\ell \\ C_i \in \psi_2^{-1}(C'_i)}} 2^{\ic(C_1) + \cdots + \ic(C_\ell)} = \sum_{T \in \psi_2^{-1}(T')} 2^{\sum_{C \in \mc(T)} \ic(C)} \\
        &= \sum_{T \in \psi_2^{-1}(T')} 2^{\ic(T)} \qquad\text{(by \cref{lemm:ic-into-smaller-parts})}.
    \end{align*}
\end{proof}

\subsection{Proof of \texorpdfstring{\cref{theo:psi-is-correct-map}}{Theorem 5.2}}
\begin{proof}[Proof of \cref{theo:psi-is-correct-map}]
    Let $T'$ be a triple-free GT pattern with bottom row $k$.
    By applying \cref{prop:ic-in-alpha-and-beta,prop:omega-in-ic} successively, we obtain
    \begin{equation*}
        \sum_{x} 2^{\ic(x)}
        = \sum_{T} 2^{\ic(T)}
        = 2^{-n} \omega_{\FSA}(T'),
    \end{equation*}
    where $x$ runs over all mixed $6$V configurations on the domain~$\mathcal{M}_\mathbf{k}$, and $T$ runs over all monotone triangles with entries in $\{1 < \overline{1} < 2 < \overline{2} < \cdots < k_n-1 < \overline{k_n-1} < k_n \}$ and with bottom row~$\mathbf{k}$ such that $\psi_2(T) = T'$.
\end{proof}

\section{Proofs of the main results} \label{sec:proof-of-main-results}

\begin{proof}[Proof of \cref{theo:extended-20V-almost-monotone-triangle-correspondence}]
    By composing the probabilistic bijections in \cref{prop:twenty-mixed6V-prob-bij,coro:mixed6V-AGTP-prob-bij}, using \cref{lemma:probabilistic-bijection-composition,theo:inv-crossing}, we obtain a probabilistic bijection between the set of $20$V configurations on the domain~$\mathcal{Q}_\mathbf{k}$, equipped with the constant weight function~$1$, and the set of triple-free GT patterns with bottom row~$\mathbf{k}$, equipped with the weight~$2^{-n} \times \omega_{\FSA}(\cdot)$.
    Hence, the theorem follows from \cref{lemm:weight-preservation}.
\end{proof}

\begin{proof}[Proof of \cref{theo:free-left-boundary-20V-enumeration}]
    By \cref{theo:extended-20V-almost-monotone-triangle-correspondence}, the total number of $20$V configurations on $\mathcal{Q}_\mathbf{k}$, where $\mathbf{k}$ ranges over all strictly increasing sequences $\mathbf{k}$ of $n$ positive integers bounded above by $m+1$, is equal to $2^{-n}$ times the weighted enumeration of triple-free GT patterns with $n$ rows and strictly increasing non-negative bottom row whose entries are bounded above by $m$, with respect to the weight~$\omega_{\FSA}(\cdot)$.
    By \cite[Theorem 1 and Proposition 3]{MR4752187}, this value is given by
    \begin{equation*}
        \prod_{i=1}^n \frac{(m-n+3i+1)_{i-1} (m-n+i+1)_{i}}{(\frac{m-n+i+2}{2})_{i-1} (i)_i}.
        \qedhere
    \end{equation*}
\end{proof}

\section*{Acknowledgements}
I am very grateful to Ilse Fischer for many helpful discussions and comments on the drafts of this paper.
I also thank Marcus Sch\"{o}nfelder for helping me in finding a definition of the variant inversion number, and thank Shane Chern for helping me guess the prefactor of the frozen domains.

\bibliographystyle{alpha}
\bibliography{main}

\end{document}